\documentclass[oneside, 10pt]{amsart}
\usepackage{graphics}
\usepackage[english]{babel}
\usepackage{amssymb,amsmath,amsfonts}
\usepackage{datetime}
\usepackage{color}
\usepackage{ulem}
\usepackage{lipsum}
\usepackage{amsfonts}
\usepackage[utf8]{inputenc}
\usepackage{cite}
\RequirePackage{amscd}
\RequirePackage{epic}
\RequirePackage[all]{xy}
\RequirePackage{url}
\RequirePackage[shortlabels]{enumitem}
\usepackage{empheq}
\usepackage{epsfig}

\usepackage[english]{babel}

\usepackage[utf8]{inputenc}
\usepackage{cite}
\RequirePackage{amscd}
\RequirePackage{epic}
\RequirePackage{eepic}
\RequirePackage[all]{xy}
\RequirePackage{url}
\RequirePackage[shortlabels]{enumitem}
\usepackage{empheq}
\usepackage{nomencl}
\usepackage{epsfig}
\usepackage{graphicx}
\newcommand{\comment}[1]{}
   
    \newcommand{\set}[1]{{\left\{#1\right\}}}
\newcommand{\pa}[1]{{\left(#1\right)}}

\newcommand{\gen}[1]{{\left\langle #1\right\rangle}}
\newcommand{\abs}[1]{{\left|#1\right|}}
\newcommand{\norm}[1]{{\left\|#1\right\|}}
\newcommand{\normal}{\mathrm}
\newcommand{\G}{\mathcal{G}}
\newcommand{\A}{\mathcal{A}}
 \newcommand{\U}{\mathcal{U}}

\newcommand{\T}{\mathbb{T}}
\newcommand{\Z}{\mathbb{Z}}
\newcommand{\R}{\mathbb{R}}
\newcommand{\C}{\mathbb{C}}
\newcommand{\teta}{\theta}
\newcommand{\eps}{\varepsilon}

\renewcommand{\Im}{\operatorname{Im}}

\newcommand{\Mat}{\operatorname{Mat}}
\newcommand{\GL}{\operatorname{GL}}

\newcommand{\mat}[1]{{\begin{matrix}#1\end{matrix}}}

\newcommand{\amat}[2]{{\begin{array}{#1}#2\end{array}}}
\newcommand{\eqsys}[1]{{\left\{\begin{array}{ll}#1\end{array}\right.}}
\newcommand{\conj}{\varphi\circ R_{2\pi\alpha}\circ \varphi^{-1}}

\newcommand{\average}[1]{\int_{\T^n}#1}
\newcommand{\media}[1]{\int_\T #1}
\newcommand{\tensor}{\otimes}% prodotto tensore
\newcommand{\inderiv}[1]{#1'^{-1}}    
\newcommand{\co}[1]{\textit{#1}}%corsivo
\newcommand{\upp}[1]{\uppercase{#1}}
\newcommand{\id}{\operatorname{id}}

\RequirePackage{url}
\newtheorem{prop}{Proposition}[section]
    \newtheorem{thm}{Theorem}[section]
    \newtheorem*{thm*}{Theorem}
    \newtheorem*{cor*}{Corollary}

    \newtheorem{cor}{Corollary}
    \newtheorem{lemma}{Lemma}
    \theoremstyle{remark}
     
\newtheorem{rmk}{Remark}[section]
\theoremstyle{definition}
\newtheorem{defn}{Definition}

\numberwithin{equation}{section}
\numberwithin{defn}{section}
\numberwithin{prop}{section}
\numberwithin{cor}{section}
\numberwithin{rmk}{section}
\newcounter{tmp}
\begin{document}
\author{Jessica Elisa Massetti}\address{Astronomy and Dynamical Systems, IMCCE (UMR 8028) - Observatoire de Paris, 77 Av. Denfert Rochereau 75014 Paris, France  \\
e-mail: jessica.massetti@obspm.fr\\
 \&\\
               Università di Roma Tre - Dipartimento di Matematica e Fisica, via della Vasca Navale 84, 00154 Roma, Italie }
 
 \title[Translated tori]{Normal form à la Moser for diffeomorphisms and generalization of R\"ussmann's translated curve theorem to higher dimension
}

\begin{abstract}
We prove a discrete time analogue of 1967 Moser's normal form of real analytic perturbations of vector fields possessing an invariant, reducible, Diophantine torus; in the case of diffeomorphisms too, the persistence of such an invariant torus is a phenomenon of finite co-dimension. Under convenient non-degeneracy assumptions on the diffeomorphisms under study (torsion property for example), this co-dimension can be reduced. As a by-product we obtain generalizations of R\"ussmann's translated curve theorem in any dimension, by a technique of elimination of parameters.
\end{abstract}
\maketitle
\setcounter{tocdepth}{1} 
\tableofcontents
\section{Introduction and results}
Let $\T = \R/2\pi\Z$, $a,b\in\R, a< b$ and consider the twist map $$P : \T\times [a,b] \to \T\times\R,\quad (\teta,r)\mapsto (\teta + \alpha(r), r),$$ where $\alpha'(r)>0$: $P$ preserves circles $r = r_0$, $r_0\in [a,b]$, and \co{twist} them by an angle which increases as $r$ does. \\Moser in \cite{Moser:1962} proved that for any $r_0\in (a,b)$ such that $\alpha(r_0)$ is Diophantine, if $Q$ is an area preserving diffeomorphism sufficiently close to $P$, it has an invariant curve near $r=r_0$ on which the dynamics is conjugated to the rotation $\teta\mapsto \teta + \alpha(r_0)$.\\
In 1970, R\"ussmann generalized this fundamental result to non-conservative twist diffeomorphisms of the annulus \cite{Russmann:1970, Bost:1985, Yoccoz:Bourbaki}. He showed that the persistence of a Diophantine, invariant circle is a phenomenon of co-dimension $1$: in general the invariant curve does not persist but is translated in the normal direction. It is the "theorem of the translated curve" (see below for a precise statement).\\
As in Kolmogorov's theorem \cite{Kolmogorov:1954}, the dynamics on the translated curve can be conjugated to the same initial Diophantine rotation because of the non degeneracy (twist) of the map. Herman gave a proof of the translated curve theorem for diffeomorphisms with rotation number of constant type \cite{Herman:1983}, then generalized R\"ussmann's result in higher dimension to diffeomorphisms of $\T^n\times\R$ ($\T^n = \R^n/2\pi\Z^n$) close enough to the rotation $(\teta,r) \mapsto (\teta + 2\pi\alpha, r)$, $2\pi\alpha$ being a Diophantine vector, without assuming any twist hypothesis but introducing an external parameter in order to tune the frequency on the translated torus, yet breaking the dynamical conjugacy to the Diophantine rotation, see \cite{Yoccoz:Bourbaki}. \\ Up to our knowledge no further generalization in $\T^n\times\R^n$ of R\"ussmann's theorem has been given so far. 

The first purpose of this work is to prove a discrete-time analogue of Moser's 1967 normal form \cite{Moser:1967} of real analytic perturbations of vector fields on $\T^n\times\R^m$ possessing a quasi-periodic Diophantine,  reducible, invariant torus. The normal form will then be used to deduce a "translated torus theorem" under convenient non-degereracy assumptions. As a by-product, R\"ussmann's classic theorem will be a particular case of small dimension. While R\"ussmann and Herman consider smooth or finite differentiable diffeomorphisms, we focus here on the analytic category. Let us state the main results.
\subsection*{A normal form for diffeomorphisms}
Let $\T^n=\R^n/2\pi\Z^n$ be the $n$-dimensional torus. Let $V$ be the space of germs along $\T^n\times\set{0}$ in $\T^n\times\R^m = \set{(\teta,r)}$ of real analytic diffeomorphisms. Fix $\alpha\in\R^n$ and $A\in\Mat_m(\R)$, assuming that $A$ is diagonalizable with eigenvalues $a_1,\ldots,a_m\in\C$ different from $0$. \\Let $U(\alpha,A)$ be the affine subspace of $V$ of diffeomorphisms of the form 
\begin{equation}
P(\teta,r) = (\teta + 2\pi\alpha + O(r), A\cdot r + O(r^2)),
\label{P0}
\end{equation}
%where $\alpha\in\R^n$, $A\in\Mat_m(\R)$ is a diagonalizable matrix of eigenvalues $a_1,\ldots,a_m\in\C$ different from $0$%
 where $O(r^k)$ are terms of order $\geq k$ in $r$ which may depend on $\teta$. For these diffeomorphisms $\text{T}^n_0 = \T^n\times\set{0}$ is an invariant, reducible, $\alpha$-quasi-periodic torus whose normal dynamics at the first order is characterized by $a_1,\ldots,a_m.$ We will collectively refer to $\alpha_1,\ldots,\alpha_n$ and $a_1,\ldots,a_m$ as the \co{characteristic frequencies} or \co{characteristic numbers} of $\normal{T}^n_0$.\\
Let $\operatorname{arg} a = (\operatorname{arg} a_1,\ldots,\operatorname{arg}a_r)\in\R^r$ ($0\leq r < m$) be the vector of the arguments of those $a_i$'s having positive imaginary part, say $a_i = \rho_i e^{i\operatorname{arg}a_i}$, where $\rho>0$, $0< \arg a_i <\pi$, and assume that the characteristic numbers satisfy the following Diophantine condition for some real $\gamma,\tau > 0$ 
\begin{equation}
\abs{2\pi k\cdot\alpha - h\cdot \operatorname{arg}a -2\pi\,l}\geq \frac{\gamma}{(1 + \abs{k})^{\tau}},\quad \forall k\in\mathbb{N}^n\setminus\set{0},\forall (l,h)\in \mathbb{Z}\times\Z^r, \abs{h}\leq 2.
\label{Diophantine Diffeos}
\end{equation}

Let $\G$ be the space of germs of real analytic isomorphisms of $\T^n\times\R^m$ of the form 
\begin{equation}
G(\teta,r) = (\varphi(\teta), R_0(\teta) + R_1(\teta)\cdot r) ,
\label{G}
\end{equation}
where $\varphi$ is a diffeomorphism of the torus fixing the origin and $R_0, R_1$ are functions defined on the torus $\T^n$ with values in $\R^n$ and $\Mat_m(\R)$ respectively.

Eventually, let us define the "translation map"
\begin{equation*}
T_\lambda: \T^n\times\R^m\to\T^n\times\R^m,\quad (\teta,r)\mapsto (\beta + \teta, b + (I + B)\cdot r), 
\end{equation*}
where $\beta\in\R^n$, $b\in\R^m$ and $B\in\Mat_m(\R)$ are such that 
\begin{equation}
(A - I)\cdot b = 0,\quad  [A,B]=0,
\label{condition parameters}
\end{equation}
having denoted with $I$ the identity matrix in $\Mat_m(\R)$.\\
We will refer to translating parameters $\lambda = (\beta, b + B\cdot r)$ as \co{corrections} or \co{counter terms}, and denote with $\Lambda$ the space of such $\lambda'$s $$\Lambda = \set{\lambda = (\beta, b + B\cdot r) : (A - I)\cdot b = 0, [A,B]=0 }.$$

\begingroup
\setcounter{tmp}{\value{thm}}% store current value of theorem counter
\setcounter{thm}{0} %assign desired value to theorem counter
\renewcommand\thethm{\Alph{thm}}% locally redefine the representation of the theorem counter

\begin{thm}[Normal form]
\label{theorem Moser diffeo}
If $Q$ is sufficiently close to $P^0\in U(\alpha,A)$, there exists a unique triplet $(G, P, \lambda)\in\G\times U(\alpha,A)\times\Lambda$, close to $(\id, P^0,0)$, such that $$ Q = T_{\lambda}\circ G\circ P \circ G^{-1}.$$
\end{thm}

\endgroup

\comment{
\begin{thm}
\label{theorem Moser diffeo}
If $Q$ is sufficiently close to $P^0\in U(\alpha,A)$, there exists a unique $(G, P, \lambda)\in\G\times U(\alpha,A)\times\Lambda$, close to $(\id, P^0,0)$, such that $$ Q = T_{\lambda}\circ G\circ P \circ G^{-1}.$$
\end{thm} }

In the neighborhood of $(\id,P^0,0)$, the $\G$-orbit of all $P's\in U(\alpha,A)$ has \emph{finite} co-dimension. The proof is based on a relatively general inverse function theorem in analytic class (Theorem \ref{teorema inversa} of the Appendix).\\ The idea of proving the finite co-dimension of a set of conjugacy classes of a diffeomorphism or of a vector field has been successfully exploited by many authors. Arnold at first proved a normal form for diffeomorphisms of $\T^n$ \cite{Arnold:1961}, followed by Moser's normal forms for vector fields \cite{Moser:1966II, Moser:1967, Wagener:2010, Massetti:2015}. Among other authors we recall Calleja-Celletti-deLaLlave work on conformally symplectic systems \cite{Calleja-Celletti-delaLlave:KAM}, Chenciner's study on the bifurcation of elliptic fixed points \cite{Chenciner:1985, Chenciner:1985a, Chenciner:1988}, Herman's twisted conjugacy for Hamiltonians \cite{Fejoz:2004, Fejoz:2010} (a generalization of Arnold's work \cite{Arnold:1961}) or Eliasson-Fayad-Krikorian work around the stability of KAM tori \cite{EFK:2015}.

This technique allows us to study the persistence of an invariant torus in two steps: first, prove a normal form that does not depend on any non-degeneracy hypothesis (but that contains the hard analysis); second, reduce or eliminate the (finite dimensional) corrections by the usual implicit function theorem, using convenient non degeneracy assumptions on the system under study. 

\subsection*{A generalization of R\"ussmann's theorem}
From the normal form of Theorem \ref{theorem Moser diffeo}, we see that when $\lambda = 0$, $Q = G \circ P\circ G^{-1}$: the torus $G(\normal{T}_0^n)$ is invariant for $Q$ and the first order dynamics is given by $P\in U(\alpha,A)$. 
Conversely, whenever $\lambda = (\beta, b)$, the torus is \co{translated} and the $2\pi\alpha$-quasi-periodic tangential dynamics is \co{twisted} by the correction in $\beta$: \[Q(\varphi(\teta), R_0(\teta)) = (\beta + \varphi(\teta + 2\pi\alpha), b + R_0(\teta + 2\pi\alpha)).\]
We will loosely say that the torus $\normal{T}^n_0$ 
\begin{itemize}[leftmargin=*]
\item persists up to twist-translation, when $\lambda = (\beta, b)$
\item persists up to translation, when $\lambda = (0, b)$
\end{itemize} 

We stress the fact that Theorem \ref{theorem Moser diffeo} not only gives the tangential dynamics to the torus, but also the normal one, of which R\"ussmann's original statement is regardless:
\begin{thm*}[R\"ussmann]
\label{Russmann theorem} Let $\alpha\in\R$ be Diophantine and $P^0: \T\times [-r_0,r_0]\to \T\times\R$ be of the form $$P^0 (\teta,r)=(\teta + 2\pi\alpha + t(r) + O(r^2), (1 + A^0) r + O(r^2)),$$ where $A^0\in\R\setminus\set{-1}$,  $t(0)=0$ and $t'(r)>0$.\\  If $Q$ is close enough to $P^0$ there exists a unique analytic curve $\gamma:\T\to \R$, close to $r=0$, an analytic diffeomorphism $\varphi$ of $\T$ close to the identity and $b\in\R$, close to $0$, such that \[Q (\teta, \gamma(\teta)) = (\conj(\teta), b + \gamma(\conj(\teta))).\]
\end{thm*}

In the original statement $A^0 = 0$; to consider this case does not add any difficulty to the proof. \\
We will generalize R\"ussmann's theorem on $\T^n\times\R^n$. At the expense of conjugating $T^{-1}_\lambda\circ Q$ to a diffeomorphism $P$ whose invariant torus has different constant normal dynamics $A$, under convenient non-degeneracy conditions we can prove the existence of a twisted-translated or translated $\alpha$-quasi-periodic Diophantine torus by application of the classic implicit function theorem in finite dimension. The following results will be proved in section \ref{Russmann theorem appendix}, where a more functional statement will be given (Theorem \ref{theorem B=0} and \ref{Russmann general}). 

On $\T^n\times\R^n$, let $P\in U(\alpha,A)$, defined in expression \eqref{P0}, be such that $A$ has simple, real, non $0$ eigenvalues $a_1,\ldots,a_n$.  This hypothesis clearly implies that the only frequencies that can cause small divisors are the tangential ones $\alpha_1,\ldots,\alpha_n$, so that we only need to require the standard Diophantine hypothesis on $\alpha$. 

\begingroup
\setcounter{tmp}{\value{thm}}% store current value of theorem counter
\setcounter{thm}{1} %assign desired value to theorem counter
\renewcommand\thethm{\Alph{thm}}% locally redefine the representation of the theorem counter

\begin{thm}
If $Q$ is sufficiently close to $P^0\in U(\alpha,A)$, the torus $\normal{T}^n_0$ persists up to twist-translation.
\label{B}
\end{thm}
\endgroup

If, in addition, $Q$ has a torsion property we can prove the following theorem.
\begingroup
\setcounter{tmp}{\value{thm}}% store current value of theorem counter
\setcounter{thm}{2} %assign desired value to theorem counter
\renewcommand\thethm{\Alph{thm}}% locally redefine the representation of the theorem counter

\begin{thm} 
Let \[P^0(\teta,r) = (\teta + 2\pi\alpha + p_1(\teta)\cdot r + O(r^2), (I + A^0) \cdot r + O(r^2)),\] be such that \[\operatorname{det}\pa{\int_{\T^n} p_1(\teta)\,d\teta}\neq 0.\]
If $Q$ is sufficiently close to $P^0$, the torus $\normal{T}^n_0$ persists up to translation.
\label{C}
\end{thm}

\endgroup

The paper is organized as follows: in sections \ref{section normal form operator}-\ref{section difference equations} we introduce the normal form operator, define conjugacy spaces and present the difference equations that will be solved to linearize the dynamics on the perturbed torus; in section \ref{inversion of phi} we will prove Theorem \ref{theorem Moser diffeo} while in section \ref{Russmann theorem appendix} we will prove Theorems \ref{B} and \ref{C}.

\section{The normal form operator}
\label{section normal form operator}
We will show that the operator 
\begin{equation*}
\phi: \G\times U(\alpha,A)\times \Lambda \to V,\quad (G, P, \lambda)\mapsto  T_{\lambda}\circ G\circ P \circ G^{-1}
\end{equation*}

is a local diffeomorphism (in the sense of scales of Banach spaces) in a neighborhood of $(\id, P^0, 0)$.  Note that $\phi$ is formally defined on the whole space but $\phi (G, P,\lambda)$ is analytic in the neighborhood of $\normal{T}^n_0$ only if $G$ is close enough to the identity with respect to the width of analyticity of $P$. See section \ref{section operator}.

 Although the difficulty to overcome in the proof is rather standard for conjugacy problems of this kind (proving the fast convergence of a Newton-like scheme), the procedure relies on a relatively general inverse function theorem (Theorem \ref{teorema inversa} of section \ref{section implicit}), following a strategy alternative to Zehnder's in \cite{Zehnder:1975}. Both Zehnder's approach and ours rely on the fact that the fast convergence of the Newton' scheme is somewhat independent of the internal structure of the variables.  

\subsection{Complex extensions}
Let us extend the tori \[\T^n=\R^n/{2\pi\Z^n}\qquad \text{and}\qquad \normal{T}^n_0=\T^n\times\set{0}\subset\T^n\times\R^m,\] as \[\T^n_{\C}= \C^n/{2\pi\Z^n}\qquad \text{and} \qquad \text{T}^n_\C = \T^n_{\C}\times\C^m\] respectively, and consider the corresponding $s$-neighborhoods defined using $\ell^\infty$-balls (in the real normal bundle of the torus): \[\T^n_s=\set{\teta\in\T^n_{\C}:\, \max_{1\leq j \leq n}\abs{\Im\teta_j}\leq s}\quad\text{and}\quad \text{T}^n_s=\set{\pa{\teta,r}\in\text{T}^n_{\C}:\, \abs{(\Im\teta,r)}\leq s},\]
where $\abs{(\Im\teta,r)}:= \max_{1\leq j\leq n}\max(\abs{\Im\teta_j},\abs{r_j})$. \\

Let now $f: \normal{T}^n_s\to \C$ be real holomorphic on the interior of $\normal{T}^n_s$, continuous on $\normal{T}^n_s$, and consider its Fourier expansion $f(\teta,r)=\sum_{k\in\Z^n}\,f_k(r)\,e^{i\,k\cdot\teta}$, noting $k\cdot\teta = k_1\teta_1 +\ldots + k_n\teta_n$. In this context we introduce the so called "weighted norm": \[\abs{f}_s := \sum_{k\in\Z^n}\abs{f_k}\, e^{\abs{k}s},\quad \abs{k}=\abs{k_1}+\ldots +\abs{k_n},\] $\abs{f_k}=\sup_{\abs{r}<s} \abs{f_k(r)}$. Whenever $f : \normal{T}^n_s\to\C^{n}$, $\abs{f}_s = \max_{1\leq j\leq n}(\abs{f_j}_s)$, $f_j$ being the $j$-th component of $f(\teta,r)$. \\ It is a trivial fact that the classical sup-norm is bounded from above by the weighted norm: \[\sup_{z\in{\normal{T}^n_s}}\abs{f(z)}\leq\abs{f}_s\] and that $\abs{f}_s<+\infty$ whenever $f$ is analytic on its domain, which necessarily contains some $\normal{T}^n_{s'}$ with $s'>s$.\footnote{The inequality shows the well known fact that if $f$ is real analytic on $\T^n$, it admits a holomorphic bounded extension: its Fourier's coefficients decay esponentially and there exists $s>0$ such that $\abs{f}_s<\infty$}  In addition, the following useful inequalities hold if $f,g$ are analytic on $\normal{T}^n_{s'}$ \[\abs{f}_s\leq\abs{f}_{s'}\,\text{ for }\, 0<s<s',\] and \[\abs{fg}_{s'}\leq \abs{f}_{s'}\abs{g}_{s'}.\] For more details about the weighted norm, see for example \cite{Meyer:1975}.\\
In general for complex extensions $U_s$ and $V_{s'}$, we will denote by $\mathcal{A}(U_s,V_{s'})$ the set of holomorphic functions from $U_s$ to $V_{s'}$ and $\mathcal{A}(U_s)$, endowed with the $s$-weighted norm, the Banach space $\mathcal{A}(U_s,\C)$.

Eventually, let $E$ and $F$ be two Banach spaces,

\begin{itemize}[leftmargin=*]
\item We indicate contractions with a dot "$\,\cdot\,$", with the convention that if $l_1,\ldots, l_{k+p}\in E^\ast$ and $x_1,\ldots, x_p\in E$ 
\begin{equation*}
(l_1\tensor\ldots\tensor l_{k+p})\cdot (x_1\tensor\ldots\tensor x_p) = l_1\tensor\ldots \tensor l_k \gen{l_{k+1},x_1}\ldots \gen{l_{k+p},x_p}.
\end{equation*}
In particular, if $l\in E^\ast$, we simply write $l^n= l\tensor\ldots\tensor l$.\\
\item If $f$ is a differentiable map between two open sets of $E$ and $F$, $f'(x)$ is considered as a linear map belonging to $F\tensor E^{\ast}$,  $f'(x): \zeta\mapsto f'(x)\cdot\zeta$; the corresponding norm will be the standard operator norm \[\abs{f'(x)} = \sup_{\zeta\in E, \abs{\zeta}_E=1}\abs{f'(x)\cdot\zeta}_F.\]
\end{itemize}

\subsection{Spaces of conjugacies}
\begin{itemize}[leftmargin=*]
\item We consider the set $\G^{\sigma}_s$ of germs of holomorphic diffeomorphisms on $\normal{T}^n_s$ such that $$\abs{\varphi - \id}_s \leq \sigma$$ and $$\abs{R_0 + (R_1 - \id)\cdot r}_s\leq\sigma,$$
and endow the tangent space at the identity $T_{\id}\G^{\sigma}_s$ with the norm $$\abs{\dot G}_s = \max_{1\leq j \leq n+m} \pa{\abs{\dot{G}_j}_s}.$$
\begin{figure}[h!]
\begin{picture}(0,0)%
\includegraphics{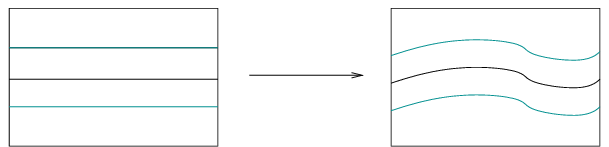}%
\end{picture}%
\setlength{\unitlength}{1657sp}%
\begingroup\makeatletter\ifx\SetFigFont\undefined%
\gdef\SetFigFont#1#2#3#4#5{%
  \reset@font\fontsize{#1}{#2pt}%
  \fontfamily{#3}\fontseries{#4}\fontshape{#5}%
  \selectfont}%
\fi\endgroup%
\begin{picture}(7905,1881)(1561,-5473)
\put(5581,-4426){\makebox(0,0)[lb]{\smash{{\SetFigFont{6}{7.2}{\familydefault}{\mddefault}{\updefault}{\color[rgb]{0,0,0}$G$}%
}}}}
\put(1846,-3751){\makebox(0,0)[lb]{\smash{{\SetFigFont{6}{7.2}{\familydefault}{\mddefault}{\updefault}{\color[rgb]{0,0,0}$\text{T}_{s+\sigma}$}%
}}}}
\put(1576,-4246){\makebox(0,0)[lb]{\smash{{\SetFigFont{6}{7.2}{\familydefault}{\mddefault}{\updefault}{\color[rgb]{0,.56,.56}$\text{T}_s$}%
}}}}
\put(1711,-4606){\makebox(0,0)[lb]{\smash{{\SetFigFont{6}{7.2}{\familydefault}{\mddefault}{\updefault}{\color[rgb]{0,0,0}$\text{T}_{0}$}%
}}}}
\put(9451,-4201){\makebox(0,0)[lb]{\smash{{\SetFigFont{6}{7.2}{\familydefault}{\mddefault}{\updefault}{\color[rgb]{0,.56,.56}$G(\text{T}_s)$}%
}}}}
\put(9406,-4651){\makebox(0,0)[lb]{\smash{{\SetFigFont{5}{6.0}{\familydefault}{\mddefault}{\updefault}{\color[rgb]{0,0,0}$G(\text{T}_0)$}%
}}}}
\end{picture}%

\caption{Deformed complex domain}
\end{figure}

\item Let $V_s$ be the subspace of $\mathcal{A}(\text{T}^n_s, \T^n_{\C}\times\C^m)$ of diffeomorphisms $$ Q : (\teta,r)\mapsto (f(\teta,r), g(\teta,r)),$$ where $f\in\mathcal{A}(\text{T}^n_s, \C^n), g\in\mathcal{A}(\text{T}^n_s, \C^m)$, endowed with the norm \[\abs{Q}_s = \max \pa{\abs{f}_s, \abs{g}_s}.\]
\item Let $U_s(\alpha,A)$ be the subspace of $V_s$ of those diffeomorphisms $P$ of the form $$P(\teta,r)= (\teta + 2\pi\alpha + O(r), A\cdot r + O(r^2)). $$ We will indicate with $p_i$ and $P_i$ the coefficients of the order-$i$ term in $r$, in the $\teta$ and $r$-directions respectively. 
\item If $G\in \G_{s}^\sigma$ and $P$ is a diffeomorphism over $G(\normal{T}^n_s)$ we define the following deformed norm $$ \abs{P}_{G,s}:= \abs{P\circ G}_s, $$ depending on $G$; this in order not to shrink artificially the domains of analyticity. The problem, in a smooth context, may be solved without changing the domain, by using plateau functions.
\end{itemize}
\subsection{The normal form operator}
\label{section operator}
By Theorem \ref{theorem well def} and Corollary \ref{cor well def} the following operator
\begin{equation}
\amat{llcl}{
\phi:\, & \G_{s+\sigma}^{\sigma/n}\times U_{s+\sigma}(\alpha,A)\times\Lambda &\to & V_s \\ & (G,P,\lambda) &\mapsto & T_{\lambda}\circ G\circ P \circ G^{-1} }
\label{Moser operator diffeo}
\end{equation}
is now well defined. It would be more appropriate to write $\phi_{s,\sigma}$ but, since these operators commute with source and target spaces, we will refer to them simply as $\phi$. We will always assume that $0<s<s+\sigma<1$ and $\sigma<s$.

\section{Difference equations}
\label{section difference equations}
We will apply the following Lemmata to linearize the tangent and the normal dynamics of the torus (see section \ref{inversion of phi}).\\
Let $\alpha\in\R^n$ and $\arg a = (\operatorname{arg} a_1,\ldots,\operatorname{arg}a_r)\in\R^r$ ($0\leq r < m$), the vector of arguments of complex eigenvalues of $A\in\Mat_m(\R)$ with positive imaginary part, satisfy the following conditions, which all follow from \eqref{Diophantine Diffeos}.
\begin{align}
\label{Dio 1}
&\abs{k\cdot\alpha - l}\geq \frac{\gamma}{\abs{k}^\tau},\qquad \forall k\in\mathbb{N}^n\setminus\set{0},\forall l\in\Z\\
\label{Dio 2}
&\abs{2\pi k\cdot\alpha - \operatorname{arg} a_j - 2\pi\,l}\geq \frac{\gamma}{\pa{1 + \abs{k}}^\tau},\qquad \forall k\in\mathbb{N}^n,\forall l\in\Z, j=1,\dots,r,\\
\label{Dio 3}
&\abs{2\pi k\cdot\alpha + h\cdot \arg a - 2\pi\,l}\geq \frac{\gamma}{\pa{1 + \abs{k}}^\tau},\qquad \forall (k,h)\in\mathbb{N}^n\times\Z^r\setminus\set{0},\forall l\in\Z, \quad\abs{h}=2.
\end{align}

The following fundamental Lemma is the heart of the proof of Theorem \ref{theorem Moser diffeo} and, more generally, of many stability results related to Diphantine rotations on the torus.
\begin{lemma}[Straightening the tangent dynamics]
\label{lemma cohomological circle} Let $\alpha\in\R$ be Diophantine in the sense of \eqref{Dio 1}. For any $g\in\mathcal{A}(\T_{s+\sigma})$, there exists a unique $f\in\mathcal{A}(\T_{s})$ of zero average and a unique $\mu \in\R$ such that 
\begin{equation*}
\mu +  f(\teta + 2\pi\alpha) - f(\teta) = g(\teta),\quad \mu = \media{g},
\end{equation*}
satisfying
\[\abs{f}_s\leq \frac{C}{\gamma\sigma^{\tau + 1}}\abs{g}_{s+\sigma},\] $C$ being a constant depending only on $\tau$.\\
\emph{Complement}. For any $a,b\in\R^+\setminus\set{0}$, $a\neq b$, and any $g\in\mathcal{A}(\T_{s+\sigma})$ there exists a unique $f\in\mathcal{A}(\T_{s})$ such that
\begin{equation*}
 a f(\teta + 2\pi\alpha) - b f(\teta) = g(\teta),
\end{equation*}
satisfying the same kind of estimate.
\end{lemma} 
\begin{proof}
Developing in Fourier series yields
\begin{equation*}
 \sum_k ( e^{i2\pi\,k\,\alpha} - 1)f_k e^{ik\teta} = \sum_k g_k e^{i\,k\teta};
\end{equation*}
 letting $\mu = g_0$ we formally have
\begin{equation*}
f(\teta)= \sum_{k\neq 0} \frac{g_k}{ e^{i2\pi\,k\,\alpha} - 1}e^{i\,k\teta}.
\end{equation*}
To prove the estimate, remark that for any $a,b\in\R^+$
\begin{align*}
\abs{a\, e^{i2\pi\,k\,\alpha} - b}^2 &= (a - b)^2 \cos^2{\frac{2\pi k\alpha}{2}} + (a + b)^2\sin^2{\frac{2\pi k\alpha}{2}}\\&\geq (a + b)^2 \sin^2{\frac{2\pi k\alpha}{2}} = (a + b)^2\sin^2\frac{2\pi(k\alpha - l)}{2},
\end{align*}
with $l\in\Z$. Choosing $l\in\Z$ such that $-\frac{\pi}{2}\leq \frac{2\pi(k\alpha - l)}{2}\leq \frac{\pi}{2}$, we get
\begin{equation*}
\abs{e^{i2\pi\,k\,\alpha} - 1}\geq 8\pi^{-2}\abs{k\alpha - l}\geq 8\pi^{-2}\frac{\gamma}{\abs{k}^\tau}, 
\end{equation*}
by the classical inequality $\abs{\sin x}\geq \frac{2}{\pi}\abs{x}$, whenever $ -\frac{\pi}{2}\leq x\leq \frac{\pi}{2},$ and condition \eqref{Dio 1}. To get the claimed estimate is now a standard computation. We address the reader interested to optimal estimates (with $\sigma ^\tau$ instead of $\sigma^{\tau + 1}$) to \cite{Russmann:1976}. The proof of the complement is straightforward.
\end{proof}

\begin{rmk} Note that the homological equation of the complement does not involve small divisors and it can readily be solved, without losing analyticity, just bounding the denominator form above with $\abs{a - b}$. Small divisors can occur only in the case $a = b$ or $\abs{a}=\abs{b}=1$.
\end{rmk}
 Let now $\alpha\in\R^n$ and $A\in\Mat_m (\R)$ be such that $a_i\neq 1, i=1,\ldots, m$, and consider the following operator
\begin{equation*}
L_{1,A} : \mathcal{A}(\T^n_{s+\sigma},\C^m) \to \mathcal {A}(\T^n_s,\C^m),\quad f\mapsto f(\teta + 2\pi\alpha) - A\cdot f(\teta).
\label{cohomological2}
\end{equation*}
\begin{lemma}[Relocating the torus]
Let $\alpha\in\R^n$ and $A\in\Mat_m(\R)$ be a diagonalizable matrix, with eigenvalues distinct from $1$, satisfying the Diophantine condition \eqref{Dio 2}. For every $g\in\mathcal{A}(\T^n_{s+\sigma},\C^m)$, there exists a unique preimage $f\in\mathcal{A}(\T^n_{s},\C^m)$  by $L_{1,A}$. Moreover the following estimate holds
\[\abs{f}_{s}\leq \frac{C_2}{\gamma}\frac{1}{ \sigma^{n+\tau}}\abs{g}_{s+\sigma}, \] $C_2$ being a constant depending only on the dimension $n$ and the exponent $\tau$.
\label{normal lemma}
\end{lemma} 
\begin{proof}
Let us first suppose that $A$ is diagonal. \\ Expanding both sides of $L_{1,A}f = g $ we see that the Fourier coefficient of the $j$-th component of $f$ is given by \[f^j_k= \frac{g^j_k}{e^{i2\pi\,k\,\alpha} - a_j},\] and the proof is straightforward from Lemma \ref{lemma cohomological circle}, once we write those negative $a_j$ as $a_j=\abs{a_j}e^{i\pi}.$

When $A$ is diagonalizable, let $P\in\GL_m(\C)$ be the transition matrix such that $PAP^{-1}$ is diagonal. Considering $f(\teta + 2\pi\alpha) - A\cdot f(\teta)$, and left multiplying both sides by $P$, we get \[\tilde{f}(\teta + 2\pi\alpha) + PAP^{-1}\tilde{f}(\teta)=\tilde{g},\] where we have set $\tilde{g}=Pg$ and $\tilde{f}=Pf$. This equation has a unique solution with the wanted estimates following Lemma \ref{lemma cohomological circle}, once we take care for those complex $a_j = \rho_j e^{i\operatorname{arg}a_j}, j=1,\ldots, r,$ and write the denominator as $$e^{i2\pi\,k\,\alpha} - a_j = e^{i\operatorname{arg}a_j}\pa{e^{i(2\pi\,k\,\alpha - \operatorname{arg}a_j)} - \rho_{j}}.$$
Taking into account Diophantine conditions \eqref{Dio 1} and \eqref{Dio 2}, we get the wanted thesis. We just need to put $f= P^{-1}\tilde{f}$. 

\end{proof}

Eventually, consider a holomorphic function $F$ on $\T^n_{s+\sigma}$ with values in $\Mat_m(\C)$ and define the operator 
\begin{equation*}
\amat{llcl}{
L_{2,A}:\, & \A(\T^n_{s+\sigma},\Mat_m(\C)) &\to &\A(\T^n_s,\Mat_m(\C))\\
& F &\mapsto & F(\teta + 2\pi\alpha)\cdot A - A\cdot F(\teta)}.
\label{cohomological3}
\end{equation*}

\begin{lemma}[Straighten the first order dynamics] Let $\alpha\in\R^n$ and $A\in\Mat_m(\R)$ be a diagonalizable matrix, with eigenvalues distinct from $1$, satisfying  the Diophantine conditions \eqref{Dio 1} and \eqref{Dio 3} respectively. For every $G\in\A(\T^n_{s+\sigma},\Mat_m(\C))$, such that $\int_{\T^n} G^i_i \, \frac{d\teta}{(2\pi)^n}= 0$, there exists a unique $F\in\A(\T^n_{s},\Mat_m(\C))$, having zero average diagonal elements, such that the matrix equation \[F(\teta + 2\pi\alpha)\cdot A - A\cdot F(\teta) = G(\teta)\] is satisfied; moreover the following estimate holds
\[\abs{F}_s \leq \frac{C_3}{\gamma}\frac{1}{\sigma^{n+\tau}}\abs{G}_{s+\sigma},\]
$C_3$ being a constant depending only on the dimension $n$ and the exponent $\tau$. 
\label{matrix lemma}
\end{lemma}
\begin{proof}
Let $A=\operatorname{diag}(a_1,\ldots,a_m)\in\R^m$ be diagonal and $F\in\Mat_m(\C)$ be given, expanding $L_{2,A}F = G$ we get $m$ equations of the form \[a_j \pa{F^j_{j}(\teta + 2\pi\alpha) - F^j_{j}(\teta)} = G^j_{{j}},\quad j=1,\ldots, m\] and $m^2-m$ equations of the form \[a_j F^{i}_j(\teta + 2\pi\alpha) - a_i F^{i}_{j}(\teta) = G^i_j(\teta),\quad \forall i\neq j, i,\,j = 1,\ldots, m.\] 

where we denoted $F^i_j$ the element corresponding to the $i$-th line and $j$-th column of the matrix $F(\teta)$. Taking into account condition \eqref{Dio 1}, the thesis follows from the same computations of Lemma \ref{normal lemma}.\\
Eventually, to recover the general case, we consider the transition matrix $P\in\GL_m(\C)$ and the equation \[(P F (\teta + 2\pi\alpha) P^{-1} P A P^{-1}) - P A P^{-1}P F(\teta)P^{-1} = PGP^{-1};\] letting $\tilde{F}=PFP^{-1}$ and $\tilde{G}=PGP^{-1}$, the equation is of the previous kind and by conditions \eqref{Dio 1}-\eqref{Dio 3}, via the same kind of calculations, we get the thesis. It remains to recover $G=P^{-1}\tilde{G}P$.
\end{proof}

\section{Inversion of the operator $\phi$}
\label{inversion of phi}
The following theorem represents the main result of this first part, from which the normal form Theorem \ref{theorem Moser diffeo} follows. \\ Let us fix $P^0\in U_s(\alpha,A)$ and note $V^{\sigma}_{s}=\set{Q\in V : \abs{Q - P^0}_s<\sigma}$ the ball of radius $\sigma$ centered at $P^0$.
\begin{thm}
\label{abstract Moser} The operator $\phi$ is a local diffeomorphism in the sense that for every $s<s+\sigma<1$ there exists $\eps>0$ and a unique $C^\infty$-map $\psi$ $$\psi : V^\eps_{s+\sigma}\to \G_{s}\times U_s(\alpha,A)\times\Lambda$$ such that $\phi\circ\psi = \id.$ Moreover $\psi$ is Whitney-smooth with respect to $(\alpha,A)$. 
\end{thm}
This result will follow from the inverse function theorem \ref{teorema inversa} and regularity propositions \ref{smoothness}-\ref{lipschitz}-\ref{Whitney}.\\ In order to solve locally $\phi(x)=y$, we use the remarkable idea of Kolmogorov and find the solution by composing infinitely many times the operator \[x = (g,u,\lambda)\mapsto x + \phi'^{-1}(x)\cdot(y - \phi(x)),\] on extensions $\normal{T}^n_{s+\sigma}$ of shrinking width.\\At each step of the induction, it is necessary that $\inderiv{\phi}(x)$ exists at an unknown $x$ (not only at $x_0$) in a whole neighborhood of $x_0$ and that $\inderiv{\phi}$ and $\phi''$ satisfy a suitable estimate, in order to control the convergence of the iterates.\\
The main step is to check the existence of a right inverse for 
\[\phi'(G,P,\lambda): T_G\mathcal{G}^{\sigma/n}_{s+\sigma}\times\overrightarrow{{U}}_{s+\sigma}\times\Lambda\to {V}_{G,s},\] if $G$ is close to the identity. We indicated with $\overrightarrow{U}$ the vector space directing $U(\alpha,A)$.
\begin{prop}
\label{linear proposition diffeo}
 There exists $\eps_0$ such that if $(G,P,\lambda)\in\G^{\eps_0}_{s+\sigma}\times  U_{s+\sigma}(\alpha,A)\times\Lambda$, for all $\delta Q \in V_{G,s+\sigma}= G^{\ast}\A(\normal{T}^n_{s+\sigma},\C^{n+m})$, there exists a unique triplet $(\delta G, \delta P, \delta\lambda)\in T_G\G_{s}\times \overrightarrow{ U}_s\times\Lambda$ such that 
\begin{equation}
\phi'(G,P,\lambda) \cdot (\delta G,\delta P,\delta \lambda) = \delta Q.
\label{linear equation diffeo}
\end{equation}
Moreover we have the following estimates 
\begin{equation}
\max (\abs{\delta G}_{s},\abs{\delta P}_{s},\abs{\delta\lambda})\leq \frac{C'}{\sigma^{\tau'}}\abs{\delta Q}_{G,s+\sigma},
\end{equation}
$C'$ being a constant possibly depending on $\abs{\pa{\pa{G-\id},P - \pa{\teta + 2\pi\alpha, A\cdot r}}}_{s+\sigma}$.
\label{proposition phi' diffeo}
\end{prop}
\begin{proof}
Differentiating with respect to $x = (G,P,\lambda)$, we have \[\delta (T_\lambda\circ G\circ P\circ G^{-1}) = T_{\delta\lambda}\circ (G\circ P\circ G^{-1}) + T'_{\lambda}\circ (G\circ P \circ G^{-1})\cdot \delta(G\circ P\circ G^{-1})\] hence \[M \cdot(\delta G\circ P + G'\circ P \cdot\delta P - G'\circ P\cdot P'\cdot G'^{-1}\cdot\delta G)\circ G^{-1} = \delta Q - T_{\delta\lambda}\circ(G\circ P \circ G^{-1}),\] where $M=\pa{\mat{I & 0\\ 0 & I + B}}$.\\
The data is $\delta Q$ while the unknowns are the "tangent vectors" $\delta P \in O(r)\times O(r^2)$, $\delta G$ (geometrically, a vector field along $G$) and $\delta\lambda \in \Lambda$.\\
Pre-composing by $G$ we get the equivalent equation between germs along the standard torus $\text{T}^n_0$ (as opposed to $G (\text{T}^n_0)$):
\begin{equation*}
M\cdot(\delta G\circ P + G'\circ P\cdot \delta P - G'\circ P\cdot P'\cdot G'^{-1}\cdot\delta G) =  \delta Q \circ G - T_{\delta\lambda}\circ G\circ P ;
\end{equation*} 
multiplying both sides by $(G'^{-1}\circ P)M^{-1}$, we finally obtain
\begin{equation}
\dot G \circ P - P'\cdot \dot G + \delta P =  \inderiv{G}\circ P \cdot M^{-1}\delta Q\circ G + \inderiv{G}\circ P\cdot M^{-1} T_{\delta\lambda} \circ G\circ P,
\label{linear equation diffeo2}
\end{equation}
where $\dot G = \inderiv{G}\cdot\delta G$.\\
Remark that the term containing $T_{\delta\lambda}$ is not constant; expanding along $r=0$, it reads  \[T_{\dot\lambda} = \inderiv{G}\circ P\cdot M^{-1}\cdot T_{\delta\lambda}\circ G \circ P = (\dot\beta + O(r), \dot b + \dot B\cdot r + O(r^2)).\]
The vector field $\dot G$ (geometrically, a germ along $\text{T}^n_0$ of tangent vector fields) reads \[\dot G(\teta,r) = (\dot\varphi(\teta), \dot R_0(\teta) + \dot R_1(\teta)\cdot r).\] The problem is now: $G,\lambda,P, Q$ being given, find $\dot G,\delta P$ and $ \dot\lambda$, hence $\delta\lambda$ and $\delta G$. \\ We are interested in solving the equation up to the $0$-order in $r$ in the $\teta$-direction, and up to the first order in $r$ in the action direction; hence we consider the Taylor expansions along $\text{T}^n_0$ up to the needed order.\\
We remark that since $\delta P = (O(r), O(r^2))$, it will not intervene in the cohomological equations given out by \eqref{linear equation diffeo2}, but will be uniquely determined by identification of the reminders.\\
Let us proceed to solve the equation \eqref{linear equation diffeo2}, which splits into the following three 
\begin{equation}
\begin{aligned}
%\label{0-order teta}
\dot\varphi(\teta + 2\pi\alpha) - \dot\varphi(\teta) + p_1 \cdot\dot R_0 &= \dot q_0 + \dot\beta \\
%\label{0-order r}
\dot R_0(\teta + 2\pi\alpha) -  A\cdot\dot R_0(\teta) &= \dot Q_0 + \dot b\\
%\label{1-order r}
\dot R_1(\teta + 2\pi\alpha)\cdot A -  A\cdot \dot R_1(\teta)  &=
 \dot Q_1 - ( 2P_2\cdot \dot R_0 + \dot R_0'(\teta + 2\pi\alpha)\cdot p_1)+ \dot B.
\end{aligned}
\label{linearized system}
\end{equation}
The first equation is the one straightening the tangential dynamics, while the second and the third ones are meant to relocate the torus and straighten the normal dynamics. \\ For the moment we solve the equations "modulo $\dot\lambda$"; eventually $\delta\lambda$ will be uniquely chosen to kill the constant component of the given terms belonging to the kernel of $A- I$ and $[A, \cdot]$ respectively, and solve the cohomological equations. \\ In the following we will repeatedly apply Lemmata \ref{lemma cohomological circle}-\ref{normal lemma}-\ref{matrix lemma} and Cauchy's inequality. Furthermore, we do not keep truck of constants - just note that they may only depend on $n$ and $\tau$ (from the Diophantine condition) and on $\abs{G -\id}_{s+\sigma}$ and $\abs{P - \pa{(\teta + 2\pi\alpha), A\cdot r)}}_{s+\sigma}$, and refer to them as $C$.
\begin{itemize}[leftmargin=*]
\item First, second equation has a solution $$ \dot R_0 = L_{1,A}^{-1}(\dot Q_0 + \dot b - \bar{b}), $$ where $$\bar{b} =  \int_\T \dot Q_0 + \dot b\,\frac{d\teta}{2\pi},$$ and $$\abs{\dot R_0}_s\leq \frac{C}{\gamma^2\sigma^{\tau + n}}\abs{\dot Q_0 + \dot b}_{s+\sigma}.$$
\item Second, we have \[\dot\varphi(\teta + 2\pi\alpha) - \dot\varphi(\teta) + p_1 \cdot \dot R_0 = \dot q_0 + \dot\beta - \bar\beta,\] where $\bar\beta = \int_{\T^n} \dot q_0 - p_1\cdot R_0 + \dot\beta \,\frac {d\teta}{(2\pi)^n},$ hence $$\dot\varphi = L_\alpha^{-1} (\dot q_0 + \dot\beta - \bar\beta), $$ satisfying \[\abs{\dot\varphi}_{s-\sigma}\leq \frac{C}{\gamma\sigma^{\tau + n +1}}\abs{\dot q_0 + \dot\beta}_{s + \sigma}\]
\item Third, the solution of equation in $\dot R_1$ is $$\dot R_1 = L_{2,A}^{-1}(\tilde{Q_1} + \dot B - \bar B),$$
hiving denoted $\tilde{Q_1}=  \dot Q_1 - ( 2P_2\cdot \dot R_0 + \dot R_0'(\teta + 2\pi\alpha)\cdot p_1),$
and $\bar{B}$ the average of $\tilde{Q}_1 + \dot{B}$. It satisfies $$\abs{\dot R_1}_{s-2\sigma}\leq \frac{C}{\gamma\sigma^{n+\tau}}\abs{\tilde{Q_1} + \dot B}_{s+\sigma}.$$
\end{itemize}
We now handle the unique choice of the correction $\delta\lambda = (\delta\beta, \delta b + \delta B\cdot r)$ given by $T_{\delta\lambda}$. If $\bar\lambda = (\bar\beta,\bar b + \bar B\cdot r)$, the map $f: \Lambda\to\Lambda,\, \delta\lambda\mapsto \bar\lambda$ is well defined. In particular when $G=\id$, $\frac{\partial f}{\partial\delta\lambda} = - \id$ and it will remain bounded away from $0$ if $G$ stays sufficiently close to the identity: let say $\abs{G -\id}_{s_0}\leq \eps_0,$ for $s_0<s$.  In particular, $-\bar\lambda$ is affine in $\delta\lambda$, the system to solve being triangular of the form $\average{a(G,\dot Q) + A(G)\cdot\delta\lambda}=0$, with diagonal close to $1$ if the smalleness condition above is assumed. Under these conditions $f$ is a local diffeomorphism and there exists unique $\delta\lambda$ such that $f(\delta\lambda)= 0$, satisfying $$\abs{\delta\lambda} \leq \frac{C}{\sigma^{\tilde\tau}}\abs{\delta Q}_{G,s+\sigma},$$ for some $\tilde\tau>1.$
We finally have  
\[\abs{\dot{G}}_{s-2\sigma} \leq \frac{C}{\gamma}\,\frac{1}{\sigma^{\tau ''}}\abs{\delta Q}_{G,s+\sigma}. \]

Now, from the definition of $\dot G = G'^{-1}\cdot\delta G$ we get $\delta G = G'\cdot\dot G$, hence similar estimates hold for $\delta G$: \[\abs{\delta G}_{s-\sigma}\leq \sigma^{-1}(1 + \abs{G-\id}_{s}) \frac{C}{\sigma^{\tau'''}}\abs{\delta Q}_{G,s+\sigma}.\] Eventually, equation \eqref{linear equation diffeo2} uniquely determines $\delta P$.\\ Up to redefining $\sigma' = \sigma/3$ and $s' = s + \sigma$, we have the wanted estimates for all $s', \sigma' : s'< s' + \sigma'$.
\end{proof}

\begin{prop}[Boundness of $\phi''$] The bilinear map $\phi''(x)$ \[\phi''(x): (T_G\G^{\sigma}_{s+\sigma}\times \overrightarrow{U}_{s+\sigma}(\alpha,A)\times \Lambda)^{\tensor 2}\to \A(\normal{T}^n_s,\normal{T}^n_\C),\]
satisfies the estimates \[\abs{\phi''(x)\cdot\delta x^{\tensor 2}}_{G,s}\leq \frac{C''}{\sigma^{\tau''}}\abs{\delta x}^2_{s+\sigma},\] $C''$ being a constant depending on $\abs{x}_{s+\sigma}$.
\label{phi''}
\end{prop}

\begin{proof}
Differentiating twice $\phi(x)$, yields 
\begin{align*}
 &-M \{[\delta G'\circ P\cdot\delta P + \delta G'\circ P\cdot\delta P + G''\circ P\cdot \delta P^2 - \pa{\delta G'\circ P + G''\circ P\cdot\delta P}\cdot P'\cdot G'^{-1}\cdot\delta G \\
&- G'\circ P\cdot\pa{\delta P'\cdot (- G'^{-1}\cdot\delta G'\cdot G'^{-1})\cdot\delta G}] \circ G^{-1} +\\
 &+[ \delta G'\circ P\cdot\delta P + \delta G'\circ P\cdot\delta P + G''\circ P\cdot \delta P^2 - \pa{\delta G'\circ P + G''\circ P\cdot\delta P}\cdot P'\cdot G'^{-1}\cdot\delta G \\
&- G'\circ P\cdot\pa{\delta P'\cdot (- G'^{-1}\cdot\delta G'\cdot G'^{-1})\cdot\delta G}]'\circ G^{-1}\cdot (- G'^{-1}\cdot\delta G)\circ G^{-1}\}. 
\end{align*}
Once we precompose with $G$, the estimate follows.
\end{proof}

Hypothesis of Theorem \ref{teorema inversa} are satisfied, hence the existence of $(G,P,\lambda)$ such that $Q=T_{\lambda}\circ G\circ P \circ G^{-1}$. Uniqueness and smoothness of the normal form follows from Propositions \ref{lipschitz} and \ref{smoothness}-\ref{Whitney}. Theorem \ref{abstract Moser} follows, hence Theorem \ref{theorem Moser diffeo}.

%%%%%%%%%%%%%%%%%%%%%% comment %%%%%%%%%%%%%%%%%%%%%%%%%%%%%%%%%%%%%%%%%
\comment{
\begin{rmk}
In Theorem \ref{theorem Moser diffeo} we did not assume any torsion property to be satisfied by diffeomorphisms $P\in U(\alpha, A)$ defined in \eqref{P0}. Suppose now that on $\T^n\times\R^m$, $m\geq n$, the diffeomorphism $P$ has a twist in the following sense:\[P(\teta,r) = (\teta + 2\pi\alpha + p_1(\teta)\cdot r + O(r^2), A\cdot r + O(r^2)),\] is such that \[\operatorname{rank}\pa{\int_{\T^n} p_1(\teta)\,d\teta}=n.\] In this case, taking advantage of this torsion hypothesis, one can avoid the translation parameter $\beta\in\R^n$ and solve the equation determining $\dot\varphi$ in system \eqref{linearized system} by a well choice of the constant part of $\dot R_0$. In fact, defining $\dot R_0 = \bar{R} + \tilde{R_0},$ where $\bar{R}=\int_{\T^n}\dot R_0\, d\teta$ and $\tilde{R_0}= \dot R_0 - \bar{R}$,  the equation \[\int_{\T^n} p_1(\teta)\cdot\bar{R}\, d\teta = \pa{\int_{\T^n} \dot q_0 - p_1\cdot\tilde{R_0}\, d\teta}\] determines $\bar{R}$. If $n=m$, $\bar{R}$ is uniquely determined.\\ Of course, the constant part of $\dot R_0$ being determined, it has to be absorbed by the translation parameter $b\in\R^m$ in the normal direction, needed to keep track of the non necessarily $0$ average of $\dot Q_0$ (see the second equation of system \eqref{linearized system}). In the following we will show that $\beta$ can be eliminated by simple application of the classic implicit function theorem.
\end{rmk}

}

\comment{\begin{prop} Let $P^0: \T^n\times\R^n\to\T^n\times\R^n$ in $U(\alpha,A)$ have a twist. If $Q$ is sufficiently close to $P^0$, there exists a unique $(G,P,\lambda)\in\G\times\U(\alpha,A)\times\Lambda(b,B)$, in a neighborhood of $(\id, P^0, 0)$ such that $Q = T_{\lambda}\circ G\circ P\circ G^{-1}$.
\end{prop} 
\begin{proof}
Up to a change of variables we can further assume that the matrix coefficient $p_1(\teta)= p_1$ is constant, which is a slighting simplifying, but non essential, hypothesis. \\ We first note that considering the family of maps $P^0_c(\teta,r) = P^0(\teta, c + r)$
\end{proof}}

%%%%%%%%%%%%%%%%%%%%%%%%%%%%% end comment%%%%%%%%%%%%%%%%%%%%%%%%%%%%%%%%%%%%%%%%%

\section{A generalization of R\"ussmann's theorem}
\label{Russmann theorem appendix}

Theorem \ref{theorem Moser diffeo} provides a normal form that does not rely on any non-degeneracy assumption; thus, the existence of a translated Diophantine, reducible torus will be subordinated to eliminating the "parameters in excess" $(\beta, B)$ using a non-degeneracy hypothesis. We will implicitly solve $B=0$ and $\beta=0$ by using the normal frequencies as \co{free} parameters and a torsion hypothesis respectively. R\"ussmann's classic result will be the immediate small dimensional case.

\comment{On $\T^n\times\R^n$, we will say that a diffeomorphism $P\in U (\alpha, I+A)$ has a twist if \[P(\teta,r) = (\teta + 2\pi\alpha + p_1(\teta)\cdot r + O(r^2), (I + A) \cdot r + O(r^2)),\] is such that \[\operatorname{det}\pa{\int_{\T^n} p_1(\teta)\,d\teta}\neq 0.\] }

\subsection*{Elimination of $B$}
Let $\Delta^s_m(\R)\subset\Mat_m(\R)$ be the open space of matrices with simple, real, non $0$ eigenvalues.  In $\T^n\times\R^m$, let us define
\begin{equation*}
\widehat{U} = \bigcup_{I + A\in\Delta^s_m(\R)} U (\alpha,I + A).
\end{equation*}
We recall that those $P's\in U(\alpha, I + A)$ are diffeomorphisms of the form 
$$ P(\teta,r)= (\teta + 2\pi\alpha + O(r), (I + A)\cdot r + O(r^2)),$$ on a neighborhood of $\T^n\times\set{0}.$

The following theorem is an intermediate, yet fundamental result to prove the translated torus Theorem \ref{C} and holds without requiring any torsion assumption on the class of diffeomorphisms.

\begin{thm}[Twisted Torus of co-dimension 1]
For every $P^0\in U_{s+\sigma}(\alpha,I + A^0)$ with $\alpha$ Diophantine, and $I + A^0\in\Delta^s_m(\R)$, there is a germ of $C^{\infty}$-maps
\begin{equation*}
\psi: V_{s+\sigma} \to \G_{s}\times \widehat{U}_{s}\times\Lambda(\beta,b),\quad Q\mapsto (G,P,\lambda),
\end{equation*} at $P^0\mapsto (\id, P^0, 0)$, such that $Q = T_\lambda\circ G\circ P\circ G^{-1}$, where $\lambda= (\beta,b)\in\R^{n+1}$.
\label{theorem B=0}
\end{thm}
\begin{cor}[Twisted torus] If $1$ does not belong to the spectrum of $I+A^0$, the translation correction $b=0$.
\label{corollary twisted}
\end{cor}

\begin{proof}
Denote $\phi_A$ the operator $\phi$, as now we want $A$ to vary. Let us identify with $\R^m$ the space of diagonal matrices in $\Mat_n(\R)$ and define the map
\begin{equation*}
\hat{\psi}:\R^m\times V_{s+\sigma}\to \G_{s}\times U_{s}(\alpha,I + A)\times\Lambda,\, (A,Q)\mapsto \hat{\psi}_A(Q):=\phi^{-1}_A(Q) = (G, P, \lambda)
\end{equation*}
in the neighborhood of $(A^0, P^0)$, such that $Q = T_{\lambda}\circ G\circ P \circ G^{-1}$. \\ By writing $P^0$ as $$P^0(\teta,r)=(\teta + 2\pi\alpha + O(r), (I + A_0 - \delta A)\cdot r + \delta A\cdot r + O(r^2)),$$ we remark that $P^0 = T_\lambda\circ P_{A}$, where $$\lambda = \pa{0, B (A) = (A^0 - A)\cdot(I + A)^{-1})}, $$ and $ P_A = (\teta + 2\pi\alpha + O(r), (I + A)\cdot r + O(r^2)),\, A = A^0 -\delta A.$\footnote{The terms $O(r)$ and $O(r^2)$ contain a factor $(1 - \delta A \cdot (1 + A_0)^{-1})$.}\\ According to Theorem \ref{theorem Moser diffeo}, $\phi_A(\id, P_A, \lambda) = P^0,$ thus locally for all $A$ close to $A^0$ we have $$\hat{\psi}(A,P^0)=(\id, P_A, B\cdot r),\quad B(A,P^0)= (A^0 - A)\cdot(I + A)^{-1} = \delta A\cdot (I + A^0 - \delta A)^{-1}$$ and, in particular $B(A^0,P^0)= 0 $ and \[\frac{\partial B}{\partial A}_{|A=A^0} = - (I + A^0)^{-1},\] which is invertible, due to hypothesis on the spectrum of $A^0$. 
Hence $A\mapsto B(A)$ is a local diffeomorphism and by the implicit function theorem (in finite dimension) locally for all $Q$ close to $P^0$ there exists a unique $\bar A$ such that $B (\bar A, Q) = 0$. It remains to define $\psi(Q)= \hat{\psi}(\bar A, Q)$.
\end{proof}
The proof of Corollary \ref{corollary twisted} is immediate, by conditions \eqref{condition parameters}.

\begin{rmk}
This twisted-torus theorem relies on the peculiarity of the normal dynamics of the torus $\normal{T}^n_0$. The direct applicability of the implicit function theorem is subordinated to the fact that no arithmetic condition is required on the characteristic (normal) frequencies; beyond that, since having simple, real eigenvalues is an open property, the needed correction $B$ is indeed a diagonal matrix, so that the number of free frequencies (parameters) is enough to solve, implicitly, $B(A)=0$. The generic case of complex eigenvalues is more delicate since one should guarantee that corrections $A^0 + \delta A$ at each step, satisfy Diophantine condition \eqref{Diophantine Diffeos}. It seems reasonable to think that one would need more parameters to control this issue, and verify that the measure of such stay positive; see \cite{Fejoz:2004}.
\end{rmk}

%%%%%%%%%%%%%%%%%%%%%%%%%%%%%%%%%%%%%%%%% comment %%%%%%%%%%%%%%%%%%%%%%%%%%%%%%%%

\comment{\section{The translated curve of R\"ussmann}
\label{Russmann theorem appendix}
The diffeomorphisms considered by R\"ussmann are of this kind in a neighborhood of $\normal{T}_0$
\begin{equation}
Q(\teta,r)= (\teta + 2\pi\alpha + t(r) + f(\teta,r), (1 +A) r + g(\teta,r)),
\label{Russmann diffeo}
\end{equation}
where $\alpha$ is Diophantine, $t(0)=0$ and $t'(r)>0$ for every $r$. This represents a perturbation of $$P^0 (\teta,r)= (\teta + \alpha + t(r), (1 +A) r),$$ for which $T_0$ is invariant and carries a rotation $2\pi\alpha$.
\begin{thm}[R\"ussmann]
\label{Russmann theorem} Fix $\alpha\in D_{\gamma,\tau}$ and $$P^0 (\teta,r)=(\teta + 2\pi\alpha + t(r) + O(r^2), (1 + A)r + O(r^2))\,\in U(\alpha,A)$$ such that $t(0)=0$ and $t'(r)>0$.\\  If $Q$ is close enough to $P^0$ there exists a unique analytic curve $\gamma\in\A(\T,\R)$, close to $r=0$, a diffeomorphism $\varphi$ of $\T$ close to the identity and $b\in\R$, close to $0$, such that \[Q (\teta, \gamma(\teta)) = (\conj(\teta), b + \gamma(\conj(\teta))).\]
\end{thm}
Actually in its original version the theorem is stated for $A=0$; to consider the more general case with $A$ close to $0$, does not bring any further difficulties. 

To deduce R\"ussmann's theorem from theorem \ref{theorem Moser diffeo} we need to get rid of the counter-terms $\beta$ and $B$.

\subsubsection{Elimination of $B$}
In order to deduce R\"ussmann's result from theorem \ref{theorem Moser diffeo} we need to reduce the number of translation terms of $T_\lambda$ to one, corresponding to the translation in the $r$-direction ($T_{\lambda= (0, b)}$). As we are not interested in keeping the same normal dynamics of the perturbed diffeomorphism $Q$, up to let $A$ vary and conjugate $Q$ to some well chosen $P_{A}$, we can indeed make the counter term $B\cdot r$ to be zero. \\ Let $\Lambda_2 = \set{\lambda = (\beta, b),\, \beta,b\in\R}$.
\begin{prop}
For every $P^0\in U_{s+\sigma}(\alpha,A_0)$ with $\alpha$ diophantine, there is a germ of $C^{\infty}$ maps
\begin{equation*}
\psi: \A(\normal{T}_{s+\sigma},\T_\C\times\C)\to \G_{s}\times U_{s}(\alpha,A)\times\Lambda_2,\quad Q\mapsto (G,P,\lambda),
\end{equation*} at $P^0\mapsto (\id, P^0, 0)$, such that $Q = T_\lambda\circ G\circ P\circ G^{-1}$.
\label{proposition B=0}
\end{prop}   
\begin{proof}
Denote $\phi_A$ the operator $\phi$, as now we want $A$ to vary. Let us write $P^0$ as $$P^0(\teta,r)=(\teta + 2\pi\alpha + O(r), (1 + A_0 - \delta A)\cdot r + \delta A\cdot r + O(r^2)),$$ and remark that \[P^0 = T_\lambda\circ P_{A},\quad \lambda = \pa{0, B\cdot r = (-\delta A + \frac{\delta A\cdot A_0}{1 + A_0})\cdot r}, \] where $ P_A = (\teta + 2\pi\alpha + O(r), (1 + A)\cdot r + O(r^2))$\footnote{The terms $O(r)$ and $O(r^2)$ contain a factor $(1 + \frac{\delta A}{1 + A_0 - \delta A})$} with $$ A = \pa{A_0 + \frac{\delta A(1 + A_0)}{1 + A_0 - \delta A}}.$$  According to theorem \ref{theorem Moser diffeo}, $\phi_A(\id, P_A, \lambda) = P^0.$ In particular \[\frac{\partial B}{\partial \delta A}_{|_{G=\id}}= -\id + \frac{A_0}{1+A_0},\] where $A_0$ is close to $0$. Hence, defining
\begin{equation*}
\hat{\psi}:\R\times \A(\normal{T}_{s+\sigma},\T_\C\times\C)\to \G_{s}\times U_{s}(\alpha,A)\times\Lambda,\, (A,Q)\mapsto \hat{\psi}_A(Q):=\phi^{-1}_A(Q) = (G, P, \lambda)
\end{equation*}
in the neighborhood of $(A_0, P^0)$, by the implicit function theorem locally for all $Q$ there exists a unique $\bar A$ such that $B (\bar A, Q) = 0$. It remains to define $\psi(Q)= \hat{\psi}(\bar A, Q)$.
\end{proof}
Whenever the interest lies on the translation of the curve and the dynamics tangential to it, do not care about the "final" $A$ and consider the situation that puts $B=0$. \\ In particular the graph of $\gamma(\teta):= R_0\circ\varphi^{-1}(\teta)$ is translated by $b$ and its dynamics is conjugated to $R_{2\pi\alpha}$, modulo the term $\beta$: \[Q(\teta, \gamma(\teta)) = (\beta + \conj(\teta), b + \gamma (\conj(\teta))).\]
}
%%%%%%%%%%%%%%%%%%%%%%%%% end comment %%%%%%%%%%%%%%%%%%%%%%%%%%%%%%%%%%5 
\subsection*{Elimination of $\beta$} If $Q$ satisfy a torsion hypothesis, the existence of a translated Diophantine torus can be proved.
\begin{thm}[Translated Diophantine torus]  
 \label{Russmann general}
 Let $\alpha$ be Diophantine. On a neighborhood of $\T^n\times\set{0}\subset \T^n\times\R^n$, let $P^0\in U(\alpha,I + A^0)$ be a diffeomorphism of the form \[P^0(\teta,r) = (\teta + 2\pi\alpha + p_1(\teta)\cdot r + O(r^2), (I + A^0) \cdot r + O(r^2)) ,\] with $I + A^0$ of simple, real non $0$ eigenvalues and such that \[\operatorname{det}\pa{\int_{\T^n} p_1(\teta)\,d\teta}\neq 0.\]
   If $Q$ is close enough to $P^0$ there exists a unique $A$, close to $A^0$, and a unique $(G,P,b)\in\G\times U(\alpha,I + A)\times \R^n$ such that $Q = T_b\circ G\circ P \circ G^{-1}.$
 \end{thm}
Phrasing the thesis, the graph of $\gamma = R_0\circ\varphi^{-1}$ is a translated torus on which the dynamics is conjugated to $R_{2\pi\alpha}$ by $\varphi$ (remember the form of $G\in\G$ given in \eqref{G}). 
Before proceeding with the proof of Theorem \ref{Russmann general}, let us consider a parameter $c\in B^n_1(0)$ (the unit ball in $\R^n$) and the family of maps defined by $Q_c (\teta, r) := Q (\teta, c + r)$ obtained by translating the action coordinates. Considering the corresponding normal form operators $\phi_c$, the parametrized version of Theorem \ref{theorem Moser diffeo} follows readily.\\
Now, if $Q_c$ is close enough to $P_c^0,$ Theorem \ref{theorem B=0} asserts the existence of $(G_c, P_c, \lambda_c)\in\G\times U(\alpha,A)\times\Lambda(\beta,b)$ such that \[Q_c = T_\lambda\circ G_c\circ P_c\circ G^{-1}_c.\] 
Hence we have a family of tori parametrized by $\tilde{c}= c + \int_{\T^n}\gamma\,\frac{d\teta}{(2\pi)^n}$,  $$Q(\teta, \tilde{c} + \tilde\gamma(\teta)) = (\beta(c) + \conj(\teta), b(c) + \tilde{c} + \tilde\gamma (\conj(\teta))),$$ where $\gamma := R_0\circ\varphi^{-1}$ and $\tilde{\gamma} = \gamma - \int_\T\gamma\,\frac{d\teta}{2\pi}.$
 
 \begin{proof}
Let $\hat{\varphi}$ be the function defined on $\T^n$ taking values in $\Mat_{n}(\R)$ that solves the (matrix of) difference equation \[\hat\varphi(\teta + 2\pi\alpha) - \hat\varphi(\teta) + p_1(\teta) = \int_{\T^n} p_1(\teta)\frac{d\teta}{(2\pi)^n},\] and let $F: (\teta,r)\mapsto (\teta + \hat{\varphi}(\teta)\cdot r,r)$. The diffeomorphism $F$ restricts to the identity at $\normal{T}^n_0$. At the expense of substituting $P^0$ and $Q$ with $F \circ P^0 \circ F^{-1}$ and $F\circ Q\circ F^{-1}$ respectively, we can assume that 
$$P^0(\teta,r)= (\teta + 2\pi\alpha + p_1\cdot r + O(r^2), (I + A^0)\cdot r + O(r^2)),\quad p_1 = \int_{\T^n} p_1(\teta)\frac{d\teta}{(2\pi)^n}.$$ The germs obtained from the initial $P^0$ and $Q$ are close to one another.\\
The proof will follow from Theorem \ref{theorem B=0} and the elimination of the parameter $\beta\in\R^n$ obstructing the rotation conjugacy.\\
In line with the previous reasoning, we want to show that the map $c\mapsto \beta (c)$ is a local diffeomorphism. It suffices to show this for the trivial perturbation $P^0_c$. The Taylor expansion of $P^0_c$ directly gives the normal form. In particular $b(c) = (I+A^0)\cdot c + O(c^2)$, while the map $c \mapsto \beta(c) = p_1\cdot c + O(c^2)$ is such that $\beta(0)=0$ and  $\beta'(0) = p_1$which is invertible by twist hypothesis, thus a local diffeomorphism. Hence, the analogous map for $Q_c$, which is a small $C^1$-perturbation, is a local diffeomorphism too and, together with Theorem \ref{theorem B=0}, there exists a unique $c\in\R^n$ and $A\in\Mat_n(\R)$, such that $(\beta, B)= (0,0).$
 \end{proof}

\begin{rmk} The theorem holds also on $\T^n\times\R^m$, with $m\geq n$, requiring that \[\operatorname{rank}\pa{\int_{\T^n} p_1(\teta)\,d\teta} = n. \] This guarantees that $c\mapsto \beta(c)$ is submersive, but $c$ solving $\beta(c)=0$ would no more be uniquely determined.
\end{rmk}
\begin{rmk} Theorem \ref{Russmann general} generalizes the classic translated curve theorem of R\"ussmann in higher dimension, in the case of normally hyperbolic systems such that $A$ has simple, real, non $0$ eigenvalues, for general perturbations. \\
 We stress the fact that if $P^0$ was of the form $$ P^0(\teta,r)=(\teta + 2\pi\alpha + O(r), I\cdot r + O(r^2)),$$ like in the original frame studied by R\"ussmann, we would need a whole matrix $B\in\Mat_n(\R)$ in order to solve the homological equations, and, having just $n$ characteristic frequencies at our disposal, it is hopeless to completely solve $B=0$ and eliminate the whole obstruction. The torus would not be just translated.  
\end{rmk}

\comment{\subsubsection{A comment about higher dimension}
In dimension higher than $2$, the analogue of R\"ussmann's theorem could not be possible: needing the matrix $B\in\Mat_m(\R)$, $m\geq 2$, to solve the third homological equation and having just $m$ characteristic exponents of $A$ that we may vary as we did in the last sections, it is hopeless to kill the whole $B$. As a consequence, the obtained surface will undergo more than a simple translation. The analogous result is stated as follows.\\

Let $U(\alpha,A)$ be the space of germs of diffeomorphisms along $\normal{T}^n_0\subset \T^n\times\R^m$ of the form \[P(\teta,r)= (\teta + 2\pi\alpha + T(r) + O(r^2), (1 + A)\cdot r + O(r^2)),\] where $A\in\Mat_m(\R)$ is a diagolanizable matrix of real eigenvalues $a_j\neq 0$ for $j=1,\ldots,m$ and $T(r)$ is such that $T(0)=0$ and $T'(r)$ is invertible for all $r\in\R^m$.\\ Let also $\G$ be the space of germs of real analytic isomorphisms of the form $$g(\teta,r) = (\varphi(\teta), R_0(\teta) + R_1(\teta)\cdot r),$$ $\varphi$ being a diffeomorphism of $\T^n$ fixing the origin, $R_0$ and $R_1$ an $\R^m$-valued and $\Mat_m(\R)$-valued functions defined on $\T^n$. \\ Let $\Lambda_{m^2} = \set{\lambda = (0, b + B\cdot r),\, b\in\R^m, B\in\Mat_m(\R)}$, where $B\in\Mat_m(\R)$ has the $(m^2 - m)$ non diagonal entries different from $0$.
\begin{thm}
Let $\alpha$ be Diophantine. If $Q$ is sufficiently close to $P^0\in U(\alpha,A_0)$, there exists a unique $(G,P,\lambda)\in \G\times U(\alpha,A)\times\Lambda_{m^2}$, close to $(\id, P^0, 0)$ such that \[Q = T_\lambda\circ G\circ P \circ G^{-1}.\]
\end{thm}
The proof follows from the generalization in dimension $\geq 2$ of theorem \ref{theorem Moser diffeo} and the elimination of parameters, which are not hard to recover. We just give some guiding remarks.
\begin{itemize}[leftmargin=*]
\item First note that the constant matrix $M$ appearing in the proof of proposition \ref{linear proposition diffeo} gets the form $M=\pa{\mat{\id & 0\\ 0 & \id + B}}$ and that equation \eqref{linear equation diffeo2} splits into three homological equations of the form
\begin{align*}
\dot\varphi(\teta + 2\pi\alpha) - \dot\varphi(\teta) + p_1 \cdot\dot R_0 &= \dot q_0 + \dot\beta \\
\dot R_0(\teta + 2\pi\alpha) - (\id + A)\cdot\dot R_0(\teta) &= \dot Q_0 + \dot b\\
\dot R_1(\teta + 2\pi\alpha)\cdot (\id + A) - (\id + A)\cdot \dot R_1(\teta)  &=
 \dot Q_1 +  \dot B.
\end{align*}
In particular,
\begin{itemize}
\item equation determining $\dot{R}_0$ is readily solved by applying lemma \ref{lemma cohomological circle} component-wise, if $A$ is diagonal. If $A$ is diagonalizable and $P\in\GL_m(\R)$ is such that $P^{-1}A P$ is diagonal, left multiply the equation by $P^{-1}$ and solve it for $\tilde { R}_0 = P^{-1} \dot R_0$ and $\tilde {Q}_0= P^{-1} \dot Q_0$ 
\item equation determining the matrix $\dot R_1$ consists, when $A$ is diagonal of eigenvalues $a_1,\ldots, a_m$, in solving $m$ equations of the form \[(1 + a_j) \pa{\dot R^j_{j}(\teta + 2\pi\alpha) - \dot R^j_{j}(\teta)} = \dot Q^j_{{j}},\quad j=1,\ldots, m\] and $m^2-m$ equations of the form \[(1+ a_j)\dot R^{i}_j(\teta + 2\pi\alpha) - (1 + a_i)\dot R^{i}_{j}(\teta) = \dot Q^i_j(\teta),\quad \forall i\neq j, i,j = 1,\ldots, m.\] If $A$ is diagonalizable, and $P\in\GL_m(\R)$ the transition matrix, left and right multiply the equation by $P^{-1}$ and $P$ respectively, then solve.
\end{itemize}
\item Eventually, the torsion property on $T(r)$ guarantees the elimination of the $\beta$-obstruction to the rotation conjugacy. The analogue of proposition \ref{theorem B=0} holds. 
\end{itemize}

}

\appendix
\section{Inverse function theorem \& regularity of $\phi$}
\label{section implicit}
We state here the implicit function theorem we use to prove Theorem \ref{theorem Moser diffeo} as well as the regularity statements needed to guarantee uniqueness and smoothness of the normal form. These results follow from Féjoz \cite{Fejoz:2010, Fejoz:2015}. Remark that we endowed functional spaces with weighted norms and bounds appearing in propositions \ref{proposition phi' diffeo}-\ref{phi''} may depend on $\abs{x}_s$ (as opposed to the analogue statements in \cite{Fejoz:2010, Fejoz:2015}); for the corresponding proofs taking account of these (slight) differences we send the reader to\cite{Massetti:2015, Massetti:2015b} and the proof or Moser's theorem therein.\smallskip\\
% for the corresponding proofs taking account of these (slight) differences we send the reader to \cite{Massetti:2015, Massetti:2015b} and the proof of Moser's theorem therein.\medskip \\
Let $E=(E_s)_{0<s<1}$ and $F=(F_s)_{0<s<1}$ be two decreasing families of Banach spaces with increasing norms $\abs{\cdot}_s$ and let $B^E_s(\sigma)=\set{x\in E : \abs{x}_s<\sigma}$ be the ball of radius $\sigma$ centered at $0$ in $E_s$. \\ On account of composition operators, we additionally endow $F$ with some deformed norms which depend on $x\in B^E_s(s)$ such that \[ \abs{y}_{0,s} = \abs{y}_s \qquad \text{and}\qquad \abs{y}_{\hat{x},s}\leq \abs{y}_{x, s + \abs{x - \hat{x}}_s}.\]
Consider then operators commuting with inclusions $\phi: B^E_{s+\sigma}(\sigma) \to F_s$, with $0<s<s+\sigma<1$, such that $\phi(0) = 0$. \\We then suppose that if $x\in B^E_{s+\sigma}(\sigma)$ then $\phi'(x):E_{s+\sigma}\to F_s$ has a right inverse $\phi'^{-1}(x): F_{s+\sigma}\to E_s$ (for the particular operators $\phi$ of this work, $\phi'$ is both left and right invertible).\\ $\phi$ is supposed to be at least twice differentiable.\\
Let $\tau:=\tau'+\tau''$ and $C:=C'C''$.
\begin{thm}[Inverse function theorem]
Further assume
\begin{align}
\label{bound phi'^-1}
\abs{\inderiv{\phi}(x)\cdot\delta y}_s\, &\leq \frac{C'}{\sigma^{\tau'}}\abs{\delta y}_{x,s+\sigma}\\
\label{bound phi''}
\abs{\phi''(x)\cdot \delta x^{\tensor 2}}_{x,s}\, &\leq \frac{C''}{\sigma^{\tau''}}\abs{\delta x}^2_{s+\sigma},\quad \forall s,\sigma: 0<s<s+\sigma<1
\end{align}
$C'$ and $C''$ depending on $\abs{x}_{s+\sigma}$, $\tau',\tau''\geq 1$. \\
For any $s, \sigma, \eta$ with $\eta<s$ and $\eps\leq \eta \, \frac{\sigma^{2\tau}}{2^{8\tau}C^2}$ ($ C\geq 1,\sigma< 3 C$), $\phi$ has a right inverse $\psi: B^F_{s+\sigma}(\eps)\to B^E_{s}(\eta)$. In other words, $\phi$ is locally surjective: $$ B^{F}_{s+\sigma}(\eps)\subset \phi(B^{E}_{s}(\eta)).$$ 
\label{teorema inversa}
\end{thm}

\comment{
Define
\begin{equation}
Q : B^E_{s+2\sigma}(\sigma) \times B^E_{s+2\sigma}\to F_s,\quad (x,\hat{x})\mapsto \phi(\hat{x}) - \phi(x) - \phi'(x)(\hat{x} - x),
\label{resto Taylor}
\end{equation}  
 the reminder of the Taylor formula.\\

\begin{lemma} 
\label{lemma Q}
For every $x,\hat{x}$ such that $\abs{x-\hat{x}}_s < \sigma$, 
\begin{equation}
\abs{Q(x,\hat{x})}_{x,s}\leq \frac{C''}{2\sigma^2}\,\abs{\hat{x} - x}^2_{{s+\sigma+\abs{\hat{x} - x}}_s}.
\label{stima Q}
\end{equation}
\end{lemma}

\begin{proof}
Let $x_t = (1-t)x + t\hat{x}$, $0\leq t\leq 1$, be the segment joining $x$ to $\hat{x}$. Using Taylor's formula,  \[Q(x,\hat{x})= \int_0^1 (1-t)\phi''(x_t)(\hat{x}-x)^2\, dt,\] hence
\begin{align*}
\abs{Q(x,\hat{x})}_{x,s}&\leq \int_0^1 (1-t)\abs{\phi''(x_t)(\hat{x}-x)^2}_{x,s}\, dt \\
&\leq \int_0^1 (1-t)\abs{\phi''(x_t)(\hat{x}-x)^2}_{x_t,s + \abs{x_t - x}_{s}}\,dt\\
&\leq \int_0^1 (1-t)\,\frac{C''}{\sigma^2} \abs{(\hat{x}-x)}^2_{s +\sigma + \abs{x_t - x}_{s}}\,dt\\
&\leq \frac{C''}{2\sigma^2}\,\abs{\hat{x} - x}^2_{{s+\sigma+\abs{\hat{x} - x}}_s}.
\end{align*}

\end{proof}

\begin{proof}[Proof of theorem \ref{teorema inversa}]
Let $s,\sigma,\eta$, with $\eta<s<1$ be fixed positive real numbers. Let also $y\in B^F_{s+\sigma}(\eps)$, for some $\eps>0$. We define the following map: 
\begin{equation*}
f: B^E_{s+\sigma}(\sigma)\to E_s,\quad x\mapsto x + \inderiv{\phi}(x)(y - \phi(x)).
\end{equation*}
We want to prove that, if $\eps$ is sufficiently small, there exists a sequence defined by induction by
\begin{equation*}
\eqsys{x_0 = 0\\
x_{n+1} = f(x_n) ,}
\end{equation*}
converging towards some point $x\in B^E_s(\eta)$, a preimage of $y$ by $\phi$.\\ Let us introduce two sequences
\begin{itemize}[leftmargin=*]
\item a sequence of positive real numbers $(\sigma_n)_{n\geq 0}$ such that $3\sum_n \sigma_n = \sigma$ be the total width of analyticity we will have lost at the end of the algorithm,\\
\item the decreasing sequence $(s_n)_{n\geq 0}$ defined inductively by $s_0 = s + \sigma$ (the starting width of analyticity), $s_{n+1}= s_n - 3\sigma_n$. Of course, $s_n\to s$ when $n\to + \infty$.
\end{itemize}  
Suppose now the existence of $x_0,...,x_{n+1}$.\\From $x_{k} - x_{k-1} =  \inderiv{\phi}(x_{k-1})(y - \phi(x_{k-1}))$ we see that $y - \phi(x_k) = - Q(x_{k-1},x_k)$, which permits to write $x_{k+1} - x_k = - \inderiv{\phi}(x_k)Q(x_{k-1},x_k)$, for $k= 1,...,n$.\\ Assuming that $\abs{x_{k} - x_{k-1}}_{s_k}\leq \sigma_k$, for $k=1,...n$, from the estimate of the right inverse and the previous lemma we get 

\begin{equation*}
\abs{x_{n+1} - x_n}_{s_{n+1}}\leq \frac{C}{2\sigma_n^{\tau}}\abs{x_n - x_{n-1}}^2_{s_{n}} \leq \ldots \leq C_n C_{n-1}^2\ldots C_1^{2^{n-1}}\abs{x_1 - x_0}^{2^n}_{s_1},
\end{equation*}
with $C_n=\frac{C}{2\sigma_n^{\tau}}$.\\ First, remark that $$\abs{x_1 - x_0}_{s_1}\leq \frac{C'}{(3\sigma_0)^{\tau'}}\,\abs{y-\phi(x_0)}_{s_0}\leq \frac{C}{2\sigma^{\tau}_0}\,\abs{y}_{s+\sigma}\leq \frac{C}{2\sigma^{\tau}_0}\,\eps.$$
Second, observe that if $C_k\geq 1$ (see remark below), $$\abs{x_{n+1}-x_n}_{s_{n+1}}\leq \pa{\eps\prod_{k\geq 0}C_k^{2^{-k}}}^{2^n}.$$
Third, note that $$ \sum_{n\geq 0} z^{2^n} = z + z^2 + z^4 + \ldots \leq z\sum_{n\geq 0} z^n \leq 2z, $$ if $z\leq\frac{1}{2}$.

The key point is to choose $\eps$ such that $\eps \prod_{k\geq 0}C_k^{2^{-k}}\leq \frac{1}{2}$ (or any positive number $< 1$)  and $\sum_{n\geq 0}\abs{x_{n+1} - x_n}_{s_{n+1}}<\eta$, in order for the whole sequence $(x_k)$ to exist and converge in $B_s(\eta)\subset E_s$. Hence, using the definition of the $C_n$'s and the fact that $$\pa{\frac{C}{2}}^{-2^{-k}}=\pa{\frac{2}{C}}^{\pa{\frac{1}{2}}^k}\Longrightarrow\prod  \pa{\frac{2}{C}}^{\pa{\frac{1}{2}}^k} = \pa{\frac{2}{C}}^{\sum\frac{1}{2^k}}=\pa{\frac{2}{C}}^2,$$ within $\sum_k \frac{1}{2^k} = \sum_k k\frac{1}{2^{k}}=2$, we obtain as a sufficient value
\begin{equation}
\eps = \eta \frac{2}{C^2}\prod_{k\geq 0}\sigma_k^{\tau\,(\frac{1}{2})^k}.
\label{bound epsilon}
\end{equation}
Eventually, the constraint $3\sum_{n\geq 0}\sigma_n = \sigma$ gives $\sigma_k= \frac{\sigma}{6}\,\pa{\frac{1}{2}}^k$, which, plugged into \eqref{bound epsilon}, gives: 
\[\eps = \eta \, \frac{2}{C^2}\pa{\frac{\sigma}{12}}^{2\tau} > \frac{\sigma^{2\tau}\eta}{2^{8^{\tau}}C^2},\]
hence the theorem. 

A posteriori, the exponential decay we proved makes straightforward the further assumption $\abs{x_k - x_{k-1}}_{s_k}<\sigma_k$ to apply lemma \ref{lemma Q}.\\ Concerning the bounds over the constant $C$, as $\sum_k \abs{x_{k+1} - x_k}_{s_{k+1}}\leq\eta$, we see that all the $\abs{x_n}_{s_n}$ are bounded, hence the constants $C'$ and $C''$ depending on them. \\Moreover, to have all the $C_n\geq 1$, as we previously supposed, it suffices to assume $C\geq \sigma/3$.
\end{proof}

\subsection{Uniqueness and regularity of the normal form}

\begin{defn}
We will say that a family of norms $(\abs{\cdot}_s)_{s>0}$ on a grading $(E_s)_{s>0}$ is \emph{log-convex} if for every $x\in E_s$ the map $ s\mapsto \log\abs{x}_s $ is convex.
\end{defn}

\begin{lemma} If $(\abs{\,\cdot\,}_s)$ is log-convex, the following inequality holds \[\abs{x}^2_{s+\sigma}\leq \abs{x}_s \abs{x}_{s+\tilde{\sigma}},\quad \forall s,\sigma,\tilde{\sigma}=\sigma(1+\frac{1}{s}).\] 
\end{lemma}

\begin{proof}
If $f: s\mapsto \log\abs{x}_s$ is convex, this inequality holds \[f\pa{\frac{s_1 + s_2}{2}}\leq \frac{f(s_1) + f(s_2)}{2}.\] Let now $x\in E_s$, then
\[\log\abs{x}_{s+\sigma}\leq\log\abs{x}_{\frac{2s + \tilde{\sigma}}{2}}\leq \frac{1}{2}\pa{\log\abs{x}_s + \log\abs{x}_{s+\tilde{\sigma}}}=\frac{1}{2}\log(\abs{x}_s\abs{x}_{s+\tilde{\sigma}}),\]hence the lemma. 
\end{proof}
Let us assume that the family of norms $(\abs{\cdot}_s)_{s>0}$ of the grading $(E_s)_{s>0}$ are log-convex. To prove the uniqueness of $\psi$ we are going to assume that $\phi'$ is also left-invertible. 

}

\begin{prop}[Lipschitz continuity of $\psi$] 
\label{lipschitz}
Let $\sigma<s$. If $y,\hat{y}\in B^{F}_{s+\sigma}(\eps)$ with $\eps = {3^{-4\tau}2^{-16\tau}}\frac{\sigma^{6\tau}}{4C^3}$, the following inequality holds
\[\abs{\psi(y) - \psi(\hat{y})}_s\leq L \abs{y-\hat{y}}_{x,s+\sigma},\] with $L=2C'/\sigma^{\tau'}$. In particular, $\psi$ being the unique local right inverse of $\phi$, it is also its unique left inverse.
\end{prop}

\comment{
\begin{proof}
In order to get the wanted estimate we introduce an intermediate  parameter $\xi$, that will be chosen later, such tat $\eta<\xi<\sigma<s<s+\sigma$.\\ 
To lighten notations let us call $\psi(y)=: x$ and $\psi(\hat{y})=: \hat{x}$. Let also $\eps=\frac{\xi^{2\tau}\eta}{2^{8\tau}C^2}$ so that if $y,\hat{y}\in B^{F}_{s+\sigma}(\eps)$, $x,\hat{x}\in B^E_{s+\sigma -\xi}(\eta)$, by theorem \ref{teorema inversa}, provided that $\eta< s + \sigma -\xi$ - to check later. In particular, we assume that any $x,\hat{x}\in B^{E}_{s + \sigma -\xi}$ satisfy $\abs{x - \hat{x}}_{s+\sigma - \xi}\leq 2\eta.$
Writing \[(x-\hat{x})=\phi'^{-1}(x)\cdot\phi(x)(x-\hat{x}),\]
and using $$\phi'(x)(x-\hat{x})= \phi(\hat{x})-\phi(\hat{x})-Q(x,\hat{x}),$$ we get \[ x-\hat{x}= \inderiv{\phi}(x)\pa{\phi(\hat{x}) - \phi(x) - Q(x,\hat{x})}.\]
Taking norms we have
 \begin{align*}
\abs{x-\hat{x}}_s  &\leq \frac{C'}{\sigma^{\tau'}}\abs{y - \hat{y}}_{x,s+\sigma} + \frac{C}{2\xi^{\tau}}\abs{x-\hat{x}}^2_{s+2\xi + \abs{x-\hat{x}}_{s + \xi}},\\
&\leq \frac{C'}{\sigma^{\tau'}}\abs{y - \hat{y}}_{x,s+\sigma} + \frac{C}{2\xi^{\tau}}\abs{x-\hat{x}}^2_{s+2\xi + 2\eta},
\end{align*} 
 by lemma \ref{lemma Q} and the fact that $\abs{x-\hat{x}}_{s+\xi}\leq\abs{x-\hat{x}}_{s+\sigma-\xi}$ (choosing $\xi$ so that $2\xi< \sigma $ too). \\ Let us define $\tilde{\sigma}=(2\xi + 2\eta)(1 + 1/s)$ and use the interpolation inequality \[\abs{x-\hat{x}}^2_{s+2\eta +2\xi}\leq\abs{x-\hat{x}}_s\abs{x-\hat{x}}_{s+\tilde{\sigma}}\] to obtain \[(1 - \frac{C}{2\xi^\tau}\abs{x-\hat{x}}_{s+\tilde{\sigma}})\abs{x-\hat{x}}_s\leq \frac{C'}{\sigma^{\tau'}}\abs{y-\hat{y}}_{x,s+\sigma}.\]
We now choose $\eta$ so small to have
\begin{itemize}[leftmargin=*]
\item $\tilde{\sigma}\leq\sigma - \xi$, which implies $\abs{x-\hat{x}}_{s+\tilde{\sigma}}\leq 2\eta$. It suffices to have $\eta\leq\frac{\sigma}{2(1+\frac{1}{s})} - \frac{3}{2}\xi$.
\item $\eta\leq\frac{\xi^{\tau}}{2C}$ in order to have $\frac{C}{2\xi^\tau}\abs{x-\hat{x}}_{s+\sigma}\leq\frac{1}{2}.$
\end{itemize}
A possible choice is $\xi = \frac{\sigma^2}{12}$ and $\eta = \pa{\frac{\sigma}{12}}^{2\tau}\frac{1}{4C},$ hence our choice of $\eps$.
\end{proof}
}

\begin{prop}[Smooth differentiation of $\psi$]
\label{smoothness}
 Let $\sigma<s<s+\sigma$ and $\eps$ as in proposition \ref{lipschitz}. There exists a constant $K$ such that for every $y,\hat{y}\in B_{s+\sigma}^F(\eps)$ we have \[\abs{\psi(\hat{y}) - \psi(y) - \inderiv{\phi}(\psi(y))(\hat{y} - y)}_s\leq K(\sigma) \abs{\hat{y} - y}^2_{x,s+\sigma},\] and the map $\psi': B_{s+\sigma}^F(\eps)\to L(F_{s+\sigma},E_s)$ defined locally by $\psi'(y)=\phi'^{-1}(\psi(y))$  is continuous. In particular $\psi$ has the same degree of smoothness of $\phi$. 
\end{prop}

\comment{

\begin{proof}
 Let's baptize some terms
\begin{itemize}[leftmargin=*]
\item $\Delta := \psi(\hat{y}) - \psi(y) - \inderiv{\phi}(x)(\hat{y} - y) $
\item $\delta:= \hat{y} - y$, the increment
\item $\xi:= \psi(y + \delta) - \psi(y)$
\item $\Xi:= \phi(x + \xi) - \phi(x)$.
\end{itemize}
With these new notations we can see $\Delta$ as 
\begin{align*}
\Delta &= \xi - \inderiv{\phi}(x)\cdot\Xi\\
&=\inderiv{\phi}(x)(\phi'(x)\cdot\xi - \Xi)\\
& =\inderiv{\phi}(x)(\phi'(x)\xi - \phi(x+\xi) + \phi(x))\\
& = -\inderiv{\phi}(x)Q(x,x+\xi)
\end{align*}
Taking norms we have
\[\abs{\Delta}_s\leq K \abs{\hat{y} - y}^2_{x,s+\bar{\sigma}}\] by proposition \ref{lipschitz} and lemma \ref{lemma Q}, for some $\bar\sigma$ which goes to zero when $\sigma$ does, and some constant $K>0$ depending on $\sigma$ . Up to substituting $\sigma$ for $\bar{\sigma}$, we have proved the statement.\\
In addition \[\psi'(y)= \phi^{-1}(y)'= \phi'^{-1}\circ \phi^{-1}(y)= \phi'^{-1}(\psi(y)),\] the inversion of linear operators between Banach spaces being analytic, the map $y\mapsto \phi'^{-1}(\psi(y))$ has the same degree of smoothness as $\phi'$.
\end{proof}
}

It is sometimes convenient to extend $\psi$ to non-Diophantine characteristic frequencies $(\alpha,A)$. Whitney smoothness guarantees that such an extension exists.
Let suppose that $\phi(x)=\phi_{\nu}(x)$ depends on some parameter $\nu\in B^k$ (the unit ball of $\R^k$) and that it is $C^1$ with respect to $\nu$ and that estimates on $\phi'^{-1}_{\nu}$ and $\phi_{\nu}''$ are uniform with respect to $\nu$ over some closed subset $D$ of $\R^{k}$. 

\begin{prop}[Whitney differentiability]Let us fix $\eps, \sigma, s$ as in proposition \ref{lipschitz}. The map $\psi : D\times B^{F}_{s+\sigma}(\eps)\to B^{E}_s(\eta)$ is $C^1$-Whitney differentiable and extends to a map $\psi : \R^{2n}\times B^{F}_{s+\sigma}(\eps)\to B^{E}_s(\eta) $ of class $C^1$. If $\phi$ is $C^k$, $1\leq k\leq \infty$, with respect to $\nu$, this extension is $C^k$.
\label{Whitney}
\end{prop}

\comment{

\begin{proof}
Let $y\in B^{F}_{s+\sigma}(\eps)$. For $\nu,\nu + \mu\in D$, let $x_\nu = \psi_{\nu}(y)$ and $x_{\nu + \mu}=\psi_{\nu+\mu}(y)$, implying \[\phi_{\nu+\mu}(x_{\nu+\mu}) - \phi_{\nu+\mu}(x_{\nu}) = \phi_{\nu}(x_\nu) - \phi_{\nu+\mu}(x_{\nu}).\] It then follows, since $y\mapsto\psi_{\nu+\mu}(y)$ is Lipschitz, that \[\abs{x_{\nu+\mu} - x_{\nu}}_s\leq L \abs{\phi_{\nu}(x_{\nu}) - \phi_{\nu+\mu}(x_{\nu})}_{x_{\nu},s+\sigma},\] taking ${y}=\phi_{\nu +\mu}(x_{\nu}),\hat{y}=\phi_{\nu+\mu}(x_{\nu+\mu}).$ In particular since $\nu\mapsto\phi_{\nu}(x_{\nu})$ is Lipschitz, the same is for $\nu\mapsto x_{\nu}.$
Let us now expand $\phi_{\nu+\mu}(x_{\nu+\mu})=\phi(\nu+\mu, x_{\nu+\mu})$ in Taylor at $(\nu,x_{\nu}).$ We have 
\[\phi(\nu+\mu,x_{\nu+\mu}) = \phi(\nu,x_{\nu}) + D\phi(\nu,x_{\nu})\cdot (\mu, x_{\nu+\mu} - x_{\nu}) + O(\mu^2, \abs{x_{\nu+\mu} - x_{\nu}}_s^2),\] hence formally defining the derivative $\partial_{\nu} x_{\nu} := - \phi'^{-1}_{\nu}(x_{\nu})\cdot\partial_{\nu} \phi_{\nu}(x_{\nu}),$ we obtain \[x_{\nu+\mu} - x_{\nu} - \partial_{\nu} x_{\nu}\cdot\mu = \phi'^{-1}_{\nu}(x_{\nu})\cdot O(\mu^2),\] hence \[\abs{x_{\nu+\mu} - x_{\nu} - \partial_{\nu} x_{\nu}\cdot\mu}_s = O(\mu^2)\] by Lipschitz property of $\nu\mapsto x_{\nu},$ when $\mu\mapsto 0$, locally uniformly with respect to $\nu$. Hence $\nu\mapsto x_{\nu}$ is $C^1$-Whitney-smooth and by Whitney extension theorem, the claimed extension exists. Similarily if $\phi$ is $C^k$ with respect to $\nu$, $\nu\mapsto x_{\nu}$ is $C^k$-Whitney-smooth. Ssee \cite{Abraham-Robbin:1967} for the straightforward generalization of Whitney's theorem to the case of interest to us: $\psi$ takes values in a Banach space instead of a finite dimension vector space; but note that the extension direction is of finite dimension though.
\end{proof}

}

\section{Inversion of a holomorphism of $ \T^n_s $ }
We present here a classical result and a lemma that justify definition of the normal form operator $\phi$ defined in section \ref{section operator}.\\
Complex extensions of manifolds are defined at the help of the $\ell^\infty$-norm. \\Let \[\T^n_{\C}= \C^n/{2\pi\Z^n}\qquad \text{and} \qquad \text{T}^n_\C = \T^n_{\C}\times\C^m,\]
 \[\T^n_s = \set{\teta\in\T^n_\C : \abs{\teta}:= \max_{1\leq j\leq n}\abs{\Im{\teta_j}}\leq s},\quad \text{T}^n_s=\left\{\pa{\teta,r}\in\text{T}^n_{\C}:\, \abs{(\Im\teta,r)}\leq s\right\},\]
where $\abs{(\Im\teta,r)}:= \max_{1\leq j\leq n}\max(\abs{\Im\teta_j},\abs{r_j})$.\\
Let also define $\R^n_s := \R^n \times (-s,s)$ and consider the universal covering of $\T^n_s$, $p: \R^n_s\to \T^n_s$.
\begin{thm}
\label{theorem well def} Let $v:\T^n_s\to\C^n$ be a vector field such that $\abs{v}_s < \sigma/n$. The map $\id + v:\T^n_{s-\sigma}\to\R^n_{s}$ induces a map $\varphi = \id + v : \T^n_{s-\sigma}\to\T^n_s$ which is a biholomorphism and there is a unique biholomorphism $\psi : \T^n_{s-2\sigma}\to\T^n_{s-\sigma}$ such that $\varphi\circ\psi = \id_{\T^n_{s - 2\sigma}}.$ \\
In particular the following hold: \[\abs{\psi - \id}_{s-2\sigma}\leq \abs{v}_{s-\sigma}\] and, if $\abs{v}_s<\sigma/2n$ \[\abs{\psi' - \id}_{s-2\sigma}\leq \frac{2}{\sigma}\abs{v}_{s}.\]
\end{thm}

For the proof we send again to \cite{Massetti:2015, Massetti:2015b}. 
\comment{
\begin{proof}
Let $\hat\varphi:= \id + v\circ p : \R^n_s\to\R^n_{s+\sigma}$ be the lift of $\varphi$ to $\R^n_s$.\\Let's start proving the injectivity and surjectivity of $\hat\varphi$; the same properties for $\varphi$ descend from these.
\begin{itemize}[leftmargin=*]
\item $\hat\varphi$ is injective as a map from $\R^n_{s-\sigma}\to\R^n_s$.\\Let $\hat\varphi(x)=\hat\varphi(x')$, from the definition of $\hat\varphi$ we have 
\begin{align*}
\abs{x -x'}=\abs{v\circ p(x') - v\circ p(x)} &\leq \int_0^1 \sum_{k=1}^n\abs{\partial_{x_k} \hat{v}}_{s-\sigma}\abs{x_k'-x_k}\,dt\leq \frac{n}{\sigma}\abs{v}_s\abs{x-x'}\\ & < \abs{x - x'},
\end{align*}
hence $x'=x$.
\item $\hat\varphi: \R^n_{s-\sigma}\to \R^n_{s-2\sigma}\subset \hat\varphi (\R^n_{s-\sigma})$ is surjective.\\ Define, for every $y\in\R^n_{s-2\sigma}$ the map \[f: \R^n_{s-\sigma}\to\R^n_{s-\sigma},\, x\mapsto y - v\circ p(x),\] which is a contraction (see the last but one inequality of the previous step). Hence there exists a unique fixed point such that $\hat{\varphi}(x) = x + v\circ p(x) = y.$ 
\end{itemize}
For every $k\in 2\pi\Z^n$, the function $\R^n_s\to\R^n_s,\, x\mapsto \hat\varphi(x+k) - \hat\varphi(x)$ is continuous and $2\pi\Z^n$-valued. In particular there exists $A\in\GL_n(\Z)$ such that $\hat\varphi(x + k)= \hat\varphi(x) + Ak$.
\begin{itemize}[leftmargin=*]
\item $\varphi: \T^n_{s-\sigma}\to\T^n_s$ is injective.\\ 
Let $\varphi(p(x))=\varphi(p(x')),$ with $p(x),p(x')\in\T^n_{s-\sigma}$, hence $\hat\varphi(x')=\hat\varphi(x) + k' = \hat\varphi(x + k'),$ for some $k'\in 2\pi\Z^n$, hence $x'=x + k'$, for the injectivity of $\hat\varphi$, thus $p(x)=p(x')$. In particular $\varphi$ is biholomorphic:
\begin{lemma}[\cite{Fritzsche-Grauert}] If $G\subset\C^n$ is a domain and $f: G\to\C^n$ injective and holomorphic, then $f(G)$ is a domain and $f: G\to f(G)$ is biholomorphic.
\end{lemma}
\item That $\varphi: \T^n_{s-\sigma}\to \T^n_{s-2\sigma}\subset \varphi(\T^n_{s-\sigma})$ is surjective follows from the one of $\hat\varphi$.
\item Estimate for $\psi: \T^n_{s-2\sigma}\to \T^n_{s-\sigma}$ the inverse of $\varphi$.\\ Let $\hat\psi : \R^n_{s-2\sigma}\to\R^n_{s-\sigma}$ be the inverse of $\hat\varphi$, and $y\in\R^n_{s-2\sigma}.$ From the definition of $\hat\varphi$, $v\circ p(\hat\psi(y))= y - p(\hat\psi(y)) = y - \hat\psi(y)$. Hence \[\abs{\hat\psi(y) - y}_{s-2\sigma} = \abs{v\circ p (\hat \psi (y))}_{s-2\sigma}\leq \abs{v}_{s-2\sigma}\leq\abs{v}_{s-\sigma}.\]
\item Estimate for $\psi' = \varphi'^{-1}\circ\varphi^{-1}$. We have \[\abs{\psi' - \id}_{s-2\sigma} \leq \abs{\varphi'^{-1} - \id}_{s-\sigma}\leq \frac{\abs{\varphi' - \id}_{s-\sigma}}{1 - \abs{\varphi' - \id}_{s-\sigma}}\leq \frac{2n}{2n -1}\frac{\abs{v}_s}{\sigma} \leq 2\frac{\abs{v}_s}{\sigma},\] by triangular and Cauchy inequalities.
\end{itemize}
\end{proof}

}

\begin{cor}[Well definition of the normal form operator $\phi$]
\label{cor well def}For all $s,\sigma$  if $G\in\G^{\sigma/n}_{s + \sigma}$, then $G^{-1}\in\A(\normal{T}^n_s,\normal{T}^n_{s+\sigma})$. 
\end{cor}
\begin{proof}
We recall the form of $G\in\G^{\sigma/n}_{s+\sigma}$: \[G(\teta,r)= (\varphi(\teta), R_0(\teta) + R_1(\teta)\cdot r).\]
$G^{-1}$ reads \[G^{-1}(\teta,r)=(\varphi^{-1}(\teta), R_1^{-1}\circ\varphi^{-1}(\teta)\cdot(r - R_0\circ\varphi(\teta))).\]
Up to rescaling norms by a factor $1/2$ like $\norm{x}_s := \frac{1}{2}\abs{x}$, the statement is straightforward and follows from theorem \ref{theorem well def}. By abuse of notations, we keep on indicating $\norm{x}_s$ with $\abs{x}_s$.
\end{proof}

%%%%%%%%%%%%%%%%%%%%%%%%%%%%%%%%%%%%%%%%%% Old sections %%%%%%%%%%%%%%%%%%%%%%%%%%%%%%%%%%%%%%%%%

\comment{
To deduce R\"ussmann's theorem from theorem \ref{theorem Moser diffeo} we need to get rid of the counter-terms $\beta$ and $B$.

Let $V$ be the space of germs of real analytic diffeomorphisms along $\T\times\set{0}\subset\T^n\times\R^m$ and $U(\alpha,A)$ its affine subspace of germs along $\normal{T}_0$ of real analytic diffeomorphism of the form \eqref{P^0}.\\
We introduce the set of germs of real analytic isomorphisms: 
\begin{equation}
\G = \set{G : \T\times\R\to\T\times\R: G(\teta,r) = (\varphi(\teta), R_0(\teta) + R_1\cdot r) },
\end{equation}
 $\varphi$ being a diffeomorphism of the torus fixing the origin and $R_0, R_1$ real valued functions defined on $\T$.\\ 
Finally we consider the translation function 
\begin{equation}
T_\lambda: \T\times\R\to\T\times\R,\quad (\teta,r)\mapsto (\beta + \teta, b + (1 + B)\cdot r) = (\teta,r) + \lambda,
\end{equation}
having denoted $\lambda = (\beta, b + B\cdot r)$. \\
We denote $\Lambda$ the space of translations $\Lambda = \set{\lambda = (\beta, b + B r),\, \beta,b,B\in\R}$.
\begin{thm}
\label{theorem Moser diffeo}
Let $\alpha$ be Diophantine and $P^0\in U(\alpha,A)$ be given. If $Q$ is sufficiently close to $P^0$, there exist a unique $(G, P, \lambda)\in\G\times U(\alpha,A)\times\Lambda$, close to $(\id, P^0,0)$, such that $$ Q = T_{\lambda}\circ G\circ P \circ G^{-1}.$$
\end{thm}
Whenever $\beta = 0 = B$, the curve $(\Theta, \gamma(\Theta))$, $\gamma = R_0\circ\varphi^{-1}$, is translated by $b\in\R$ and the translated curve's dynamics is conjugated to the rotation $R_{2\pi\alpha}$. R\"ussmann's theorem turns out to be a direct consequence (cf. section \ref{Russmann theorem appendix}).

The case of our interest will be when $A$ is close to $0$, as whenever the normal hyperbolicity gets large with respect to the perturbation, one can prove the actual persistence of $\normal{T}_0$ via the method of the graph transform.
\subsection{Outline of the proof}
We will show that the operator 
\begin{equation*}
\phi: \G\times U(\alpha,A)\times \Lambda \to V,\quad (G, P, \lambda)\to  T_{\lambda}\circ G\circ P \circ G^{-1}
\end{equation*} is a local diffeomorphism in a neighborhood of $(\id, P^0, 0)$. \\ Although the difficulty to overcome in this proof is rather standard for conjugacy problems of this kind (proving the fast convergence of a Newton-like scheme), it relies on a relatively general inverse function theorem (theorem \ref{teorema inversa} of section \ref{section implicit}), following an alternative strategy with respect to the one proposed by Zehnder in \cite{Zehnder:1975}. However, both Zehnder's approach and ours rely on the fact that the fast convergence of the Newton' scheme is somewhat independent of the internal structure of the variables. \\ Let us proceed with introducing the functional setting.
\subsection{Complex extensions}
Let us extend the tori \[\T=\R/{2\pi\Z}\qquad \text{and}\qquad \normal{T}_0=\T\times\set{0}\subset\T\times\R,\] as \[\T_{\C}= \C/{2\pi\Z}\qquad \text{and} \qquad \text{T}_\C = \T_{\C}\times\C\] respectively, and consider the corresponding $s$-neighborhoods defined using $\ell^\infty$-balls (in the real normal bundle of the torus): \[\T_s=\set{\teta\in\T_{\C}:\, \abs{\Im\teta}\leq s}\quad\text{and}\quad \text{T}_s=\set{\pa{\teta,r}\in\text{T}^n_{\C}:\, \abs{(\Im\teta,r)}\leq s},\]
where $\abs{(\Im\teta,r)}:= \max(\abs{\Im\teta},\abs{r})$. \\

Let now $f: \normal{T}_s\to \C$ be holomorphic, and consider its Fourier expansion $f(\teta,r)=\sum_{k\in\Z}\,f_k(r)\,e^{i\,k\cdot\teta}$. In this context we introduce the so called "weighted norm": \[\abs{f}_s := \sum_{k\in\Z}\abs{f_k}\, e^{\abs{k}s},\] $\abs{f_k}=\sup_{\abs{r}<s} \abs{f_k(r)}$. Whenever $f : \normal{T}_s\to\C^{n}$, $\abs{f}_s = \max_{1\leq j\leq n}(\abs{f_j}_s)$, $f_j$ being the $j$-th component of $f(\teta,r)$.\\ It is a trivial fact that the classical sup-norm is bounded from above by the weighted norm: \[\sup_{z\in{\normal{T}_s}}\abs{f(z)}\leq\abs{f}_s\] and that $\abs{f}_s<+\infty$ whenever $f$ is analytic on its domain, which necessarily contains some $\normal{T}_{s'}$ with $s'>s$. In addition, the following useful inequalities hold if $f,g$ are analytic on $\normal{T}_{s'}$ \[\abs{f}_s\leq\abs{f}_{s'}\,\text{ for }\, 0<s<s',\] and \[\abs{fg}_{s'}\leq \abs{f}_{s'}\abs{g}_{s'}.\] For more details about the weighted norm, see for example \cite{Meyer:1975}.\\
In general for complex extensions $U_s$ and $V_{s'}$ of $\T\times\R$, we will denote $\mathcal{A}(U_s,V_{s'})$ the set of holomorphic functions from $U_s$ to $V_{s'}$ and $\mathcal{A}(U_s)$, endowed with the $s$-weighted norm, the Banach space $\mathcal{A}(U_s,\C)$.

Eventually, let $E$ and $F$ be two Banach spaces,

\begin{itemize}[leftmargin=*]
\item We indicate contractions with a dot "$\,\cdot\,$", with the convention that if $l_1,\ldots, l_{k+p}\in E^\ast$ and $x_1,\ldots, x_p\in E$ 
\begin{equation*}
(l_1\tensor\ldots\tensor l_{k+p})\cdot (x_1\tensor\ldots\tensor x_p) = l_1\tensor\ldots \tensor l_k \gen{l_{k+1},x_1}\ldots \gen{l_{k+p},x_p}.
\end{equation*}
In particular, if $l\in E^\ast$, we simply note $l^n= l\tensor\ldots\tensor l$.\\
\item If $f$ is a differentiable map between two open sets of $E$ and $F$, $f'(x)$ is considered as a linear map belonging to $F\tensor E^{\ast}$,  $f'(x): \zeta\mapsto f'(x)\cdot\zeta$; the corresponding norm will be the standard operator norm \[\abs{f'(x)} = \sup_{\zeta\in E, \abs{\zeta}_E=1}\abs{f'(x)\cdot\zeta}_F.\]
\end{itemize}
\subsubsection{Spaces of conjugacies}
\begin{itemize}
\item We consider the set $\G^{\sigma}_s$ of germs of holomorphic diffeomorphisms on $\normal{T}_s$ such that $$\abs{\varphi - \id}_s \leq \sigma$$ as well as $$\abs{R_0 + (R_1 - \id)\cdot r}_s\leq\sigma.$$
\begin{figure}[h!]
\begin{picture}(0,0)%
\includegraphics{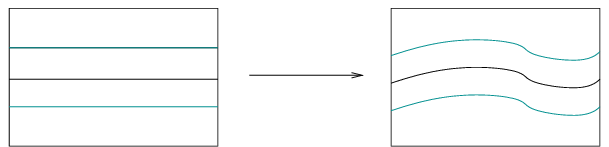}%
\end{picture}%
\setlength{\unitlength}{1657sp}%
\begingroup\makeatletter\ifx\SetFigFont\undefined%
\gdef\SetFigFont#1#2#3#4#5{%
  \reset@font\fontsize{#1}{#2pt}%
  \fontfamily{#3}\fontseries{#4}\fontshape{#5}%
  \selectfont}%
\fi\endgroup%
\begin{picture}(7905,1881)(1561,-5473)
\put(5581,-4426){\makebox(0,0)[lb]{\smash{{\SetFigFont{6}{7.2}{\familydefault}{\mddefault}{\updefault}{\color[rgb]{0,0,0}$g$}%
}}}}
\put(1846,-3751){\makebox(0,0)[lb]{\smash{{\SetFigFont{6}{7.2}{\familydefault}{\mddefault}{\updefault}{\color[rgb]{0,0,0}$\text{T}_{s+\sigma}$}%
}}}}
\put(1576,-4246){\makebox(0,0)[lb]{\smash{{\SetFigFont{6}{7.2}{\familydefault}{\mddefault}{\updefault}{\color[rgb]{0,.56,.56}$\text{T}_s$}%
}}}}
\put(1711,-4606){\makebox(0,0)[lb]{\smash{{\SetFigFont{6}{7.2}{\familydefault}{\mddefault}{\updefault}{\color[rgb]{0,0,0}$\text{T}_{0}$}%
}}}}
\put(9451,-4201){\makebox(0,0)[lb]{\smash{{\SetFigFont{6}{7.2}{\familydefault}{\mddefault}{\updefault}{\color[rgb]{0,.56,.56}$g(\text{T}_s)$}%
}}}}
\put(9406,-4651){\makebox(0,0)[lb]{\smash{{\SetFigFont{5}{6.0}{\familydefault}{\mddefault}{\updefault}{\color[rgb]{0,0,0}$g(\text{T}_0)$}%
}}}}
\end{picture}%

\caption{Deformed complex domain}
\end{figure}
\item We endow the tangent space at the identity $T_{\id}\G^{\sigma}_s$ with the norm $$\abs{\dot G}_s = \max (\abs{\dot{G_1}}_s,\abs{\dot G_2}_s)$$
\item Let $U_s(\alpha,A)$ be the subspace of $\A(\normal{T}_s,\T_\C\times\C)$ of those diffeomorphisms $P$ of the form $$P(\teta,r)= (\teta + 2\pi\alpha + O(r), (1 +A)\cdot r + O(r^2)). $$ We will indicate with $p_i$ and $P_i$ the coefficients of the order-$i$ term in $r$ in $\teta$ and $r$-directions respectively. 
\item If $G\in \G_{s}^\sigma$ and $P$ is a diffeomorphism over $G(\normal{T}_s)$ we define the following deformed norm $$ \abs{P}_{G,s}:= \abs{P\circ G}_s, $$ depending on $G$; this in order not to shrink artificially the domains of analyticity. The problem, in a smooth context, may be solved without changing the domain, by using plateau functions.
\end{itemize}
\subsubsection{The normal form operator}
\label{section operator}
Thanks to corollary \ref{cor well def} the following operator
\begin{equation}
\amat{llcl}{
\phi:\, & \G_{s+\sigma}^{\sigma}\times U_{s+\sigma}(\alpha,A)\times\Lambda &\to &\A(\normal{T}_s,\T_\C\times\C) \\ & (G,P,\lambda) &\mapsto & T_{\lambda}\circ G\circ P \circ G^{-1} }
\label{Moser operator diffeo}
\end{equation}
is now well defined. It would be more appropriate to write $\phi_{s,\sigma}$ but, since these operators commute with source and target spaces, we will refer to them simply as $\phi$. We will always assume that $0<s<s+\sigma<1$ and $\sigma<s$.
\subsubsection{Difference equation on the torus}
\begin{lemma} 
\label{lemma cohomological circle} Let $\alpha$ be Diophantine in the sense of \eqref{diophantine circle}, $g\in\mathcal{A}(\T_{s+\sigma})$ and let some constants $a,b\in\R\setminus\set{0}$ be given. There exist a unique $f\in\mathcal{A}(\T_{s})$ of zero average and a unique $\lambda\in\R$ such that the following is satisfied
\begin{equation}
\lambda + a f(\teta + 2\pi\alpha) - b f(\teta) = g(\teta),\quad \lambda = \media{g}.
\end{equation}
In particular $f$ satisfies \[\abs{f}_s\leq \frac{C}{\gamma\sigma^{\tau + 1}}\abs{g}_{s+\sigma},\] $C$ being a constant depending on $\tau$.
\end{lemma} 
\begin{proof}
Developing in Fourier series one has
\begin{equation*}
\lambda + \sum_k (a\, e^{i2\pi\,k\,\alpha} - b)f_k e^{ik\teta} = \sum_k g_k e^{i\,k\teta};
\end{equation*}
we get $\lambda = g_0 = \media{g}$ and
\begin{equation*}
f(\teta)= \sum_{k\neq 0} \frac{g_k}{a\, e^{i2\pi\,k\,\alpha} - b}e^{i\,k\teta}.
\end{equation*}
Remark that 
\begin{align*}
\abs{a\, e^{i2\pi\,k\,\alpha} - b}^2 &= (a - b)^2 \cos^2{\frac{2\pi k\alpha}{2}} + (a + b)^2\sin^2{\frac{2\pi k\alpha}{2}}\\&\geq (a + b)^2 \sin^2{\frac{2\pi k\alpha}{2}} = (a + b)^2\sin^2\frac{2\pi(k\alpha - l)}{2},
\end{align*}
with $l\in\Z$. Choosing $l\in\Z$ such that $\frac{2\pi(k\alpha - l)}{2}\in [-\frac{\pi}{2},\frac{\pi}{2}]$, we get
\begin{equation*}
\abs{a\, e^{i2\pi\,k\,\alpha} - b}\geq \frac{\pi^2}{4}\abs{a + b}\abs{k\alpha - l}\geq \frac{\pi^2}{4}\abs{a + b}\frac{\gamma}{\abs{k}^\tau}, 
\end{equation*}
using that $\abs{\sin x}\geq \frac{\pi}{2}\abs{x}, x\in [-\frac{\pi}{2},\frac{\pi}{2}]$ and condition \eqref{diophantine circle}. Hence the lemma.
\end{proof}
We address the reader interested to optimal estimates to \cite{Russmann:1976}.
\subsubsection{Inversion of $\phi'$ and bound of $\phi''$}
\begin{prop}
\label{linear proposition diffeo}
Let $0<s_0< s< s+\sigma$. There exists $\eps_0$ such that if $(G,P,\lambda)\in\G^{\eps_0}_{s_0}\times  U_{s+\sigma}(\alpha,A)\times\Lambda$, for all $\delta Q \in G^{\ast}\A(\normal{T}_{s+\sigma},\T_\C\times\C)$, there exists a unique triplet $(\delta G, \delta P, \delta\lambda)\in T_G\G_{s}\times \overrightarrow{ U_s(\alpha,A)}\times\Lambda$ such that 
\begin{equation}
\phi'(G_s,P,\lambda) \cdot (\delta G,\delta P,\delta \lambda) = \delta Q.
\label{linear equation diffeo}
\end{equation}
Moreover we have the following estimates 
\begin{equation}
\max (\abs{\delta G}_{s},\abs{\delta P}_{s},\abs{\delta\lambda})\leq \frac{C'}{\sigma^{\tau'}}\abs{\delta Q}_{G,s},
\end{equation}
$C'$ being a constant depending on $\abs{x}_{s+\sigma}$.
\label{proposition phi' diffeo}
\end{prop}
\begin{proof}
We have \[\delta (T_\lambda\circ G\circ P\circ G^{-1}) = T_{\delta\lambda}\circ (G\circ P\circ G^{-1}) + T'_{\lambda}\circ (G\circ P \circ G^{-1})\cdot \delta(G\circ P\circ G^{-1})\] hence \[M \cdot(\delta G\circ P + G'\circ P \cdot\delta P - G'\circ P\cdot P'\cdot G'^{-1}\cdot\delta G)\circ G^{-1} = \delta Q - T_{\delta\lambda}\circ(G\circ P \circ G^{-1}),\] where $M=\pa{\mat{1 & 0\\ 0 & 1 + B}}$ is $T'_\lambda$.\\
The data is $\delta Q$ while the unknowns are the "tangent vectors" $\delta P \in O(r)\times O(r^2)$, $\delta G$ (geometrically, a vector field along $g$) and $\delta\lambda \in \R^3$.\\
Pre-composing by $G$ we get the equivalent equation between germs along the standard torus $\text{T}_0$ (as opposed to $G (\text{T}_0)$):
\begin{equation*}
M\cdot(\delta G\circ P + G'\circ P\cdot \delta P - G'\circ P\cdot P'\cdot G'^{-1}\cdot\delta G) =  \delta Q \circ G - T_{\delta\lambda}\circ G\circ P ;
\end{equation*} 
multiplying both sides by $(G'^{-1}\circ P)M^{-1}$, we finally obtain
\begin{equation}
\dot G \circ P - P'\cdot \dot G + \delta P =  \inderiv{G}\circ P \cdot M^{-1}\delta Q\circ G + \inderiv{G}\circ P\cdot M^{-1} T_{\delta\lambda} \circ G\circ P,
\label{linear equation diffeo2}
\end{equation}
where $\dot G = \inderiv{G}\cdot\delta G$.\\
Remark that the term containing $T_{\delta\lambda}$ is not constant; expanding along $r=0$, it reads  \[T_{\dot\lambda} = \inderiv{G}\circ P\cdot M^{-1}\cdot T_{\delta\lambda}\circ G \circ P = (\dot\beta + O(r), \dot b + \dot B\cdot r + O(r^2)).\]
The vector field $\dot G$ (geometrically, a germ along $\text{T}_0$ of tangent vector fields) reads \[\dot G(\teta,r) = (\dot\varphi(\teta), \dot R_0(\teta) + \dot R_1(\teta)\cdot r).\] The problem is now: $G,\lambda,P, Q$ being given, find $\dot G,\delta P$ and $ \dot\lambda$, hence $\delta\lambda$ and $\delta g$. \\ We are interested in solving the equation up to the $0$-order in $r$ in the $\teta$-direction, and up to the first order in $r$ in the action direction; hence we consider the Taylor expansions along $\text{T}_0$ up to the needed order.\\
We remark that since $\delta P = (O(r), O(r^2))$, it will not intervene in the cohomological equations given out by \eqref{linear equation diffeo2}, but will be uniquely determined by identification of the reminders.\\
Let us proceed to solve the equation \eqref{linear equation diffeo2}, which splits into the following three 
\begin{align*}
%\label{0-order teta}
\dot\varphi(\teta + 2\pi\alpha) - \dot\varphi(\teta) + p_1 \cdot\dot R_0 &= \dot q_0 + \dot\beta \\
%\label{0-order r}
\dot R_0(\teta + 2\pi\alpha) - (1 + A)\dot R_0(\teta) &= \dot Q_0 + \dot b\\
%\label{1-order r}
(1 + A)\dot R_1(\teta + 2\pi\alpha) - (1 + A) \dot R_1(\teta)  &=
 \dot Q_1 - ( 2P_2\cdot \dot R_0 + \dot R_0(\teta + 2\pi\alpha)\cdot p_1)+ \dot B.
\end{align*}
The first equation is the one straightening the tangential dynamics, while the second and the third are meant to relocate the torus and straighten the normal dynamics. \\ For the moment we solve the equations "modulo $\dot\lambda$". According to lemma \ref{lemma cohomological circle}, these tree equation admit unique analytic solutions once the right hand side is average free. \begin{itemize}
\item First, second equation has a solution $$ \dot R_0 = L_\alpha^{-1}(\dot Q_0 + \dot b - \bar{b}), $$ with $$\bar{b} =  \int_\T \dot Q_0 + \dot b\,\frac{d\teta}{2\pi},$$ and $$\abs{\dot R_0}_s\leq \frac{C}{(2 + A)\gamma^2\sigma^{\tau + 1}}\abs{\dot Q_0 + \dot b}_{s+\sigma}.$$
\item Second, we have \[\dot\varphi(\teta + 2\pi\alpha) - \dot\varphi(\teta) + p_1 \cdot \dot R_0 = \dot q_0 + \dot\beta - \bar\beta,\] with $\bar\beta = \int_\T \dot q_0 - p_1\cdot R_0 + \dot\beta \,\frac {d\teta}{2\pi},$ hence $$\dot\varphi = L_\alpha^{-1} (\dot q_0 + \dot\beta - \bar\beta), $$ satisfying \[\abs{\dot\varphi}_{s-\sigma}\leq \frac{C}{\gamma\sigma^{\tau + 2}}\abs{\dot q_0 + \dot\beta}_{s + \sigma}\]
\item Third, the solution of equation in $\dot\R_1$ is $$\dot R_1 = L_\alpha^{-1}(\tilde{Q_1} + \dot B - \bar B),$$
hiving noted $\tilde{Q_1}=  \dot Q_1 - ( 2P_2\cdot \dot R_0 + \dot R_0(\teta + 2\pi\alpha)\cdot p_1),$
satisfies $$\abs{\dot R_1}_{s-\sigma}\leq \frac{C}{(2 + 2A)\gamma\sigma^{1+\tau}}\abs{\tilde{Q_1} - \dot B}_{s+\sigma}.$$
\end{itemize}
We now handle the unique choice of $\delta\lambda = (\delta\beta, \delta b + \delta B\cdot r)$ occurring in the translation map $T_{\delta\lambda}$. If $\bar\lambda = (\bar\beta,\bar b + \bar B\cdot r)$, the map $f: \Lambda\to\Lambda,\, \delta\lambda\mapsto \bar\lambda$ is well defined. In particular when $G=\id$, $\frac{\partial f}{\partial\delta\lambda} = - \id$ and it will remain bounded away from $0$ if $G$ stays sufficiently close to the identity: $\abs{G -\id}_{s_0}\leq \eps_0,$ for $s_0<s$.  In particular, $-\bar\lambda$ is affine in $\delta\lambda$, the system to solve being triangular of the form $\average{a(G,\dot v) + A(G)\cdot\delta\lambda}=0$, with diagonal close to $1$ if the smalleness condition above is assumed. Under these conditions $f$ is a local diffeomorphism and $\delta\lambda$ such that $f(\delta\lambda)= 0$ is then uniquely determined, and $$\abs{\delta\lambda} \leq \frac{C}{\sigma^{\tau +1}}\abs{\delta Q}_{G,s+\sigma}.$$

Now, from the definition of $\dot G = G'^{-1}\cdot\delta G$ we get $\delta G = G'\cdot\dot G$, hence similar estimates hold for $\delta G$: \[\abs{\delta G}_{s-\sigma}\leq \frac{C}{\sigma^{\tau + 2}}(1 + \abs{G'-\id}_{s-\sigma})\abs{\delta Q}_{G,s+\sigma}\leq \frac{C}{\sigma^{\tau + 3}}\abs{\delta Q}_{G,s+\sigma}.\] Equation \eqref{linear equation diffeo2} uniquely determines $\delta P$.\\ Up to redefining $\sigma' = \sigma/2$ and $s' = s + \sigma$, we have the wanted estimates for all $s', \sigma' : s'< s' + \sigma'$.
\end{proof}
\subsubsection{Second derivative}
We consider the bilinear map $\phi''(x)$. We have
\begin{prop}[Boundness of $\phi''$] The bilinear map $\phi''(x)$ \[\phi''(x): (T_G\G^{\sigma}_{s+\sigma}\times \overrightarrow{U_{s+\sigma}(\alpha,A)}\times \Lambda)^{\tensor 2}\to \A(\normal{T}_s,\T_\C\times\C),\]
satisfies the estimates \[\abs{\phi''(x)\cdot\delta x^{\tensor 2}}_{G,s}\leq \frac{C''}{\sigma^{\tau''}}\abs{\delta x}^2_{s+\sigma},\] $C''$ being a constant depending on $\abs{x}_{s+\sigma}$.
\label{phi''}
\end{prop}

\begin{proof}
Differentiating twice $\phi$ we get 
\begin{align*}
 &-M \{[\delta G'\circ P\cdot\delta P + \delta G'\circ P\cdot\delta P + G''\circ P\cdot \delta P^2 - \pa{\delta G'\circ P + G''\circ P\cdot\delta P}\cdot P'\cdot G'^{-1}\cdot\delta G \\
&- G'\circ P\cdot\pa{\delta P'\cdot (- G'^{-1}\cdot\delta G'\cdot G'^{-1})\cdot\delta G}] \circ G^{-1} +\\
 &+[ \delta G'\circ P\cdot\delta P + \delta G'\circ P\cdot\delta P + G''\circ P\cdot \delta P^2 - \pa{\delta G'\circ P + G''\circ P\cdot\delta P}\cdot P'\cdot G'^{-1}\cdot\delta G \\
&- G'\circ P\cdot\pa{\delta P'\cdot (- G'^{-1}\cdot\delta G'\cdot G'^{-1})\cdot\delta G}]'\circ G^{-1}\cdot (- G'^{-1}\cdot\delta G)\circ G^{-1}\}. 
\end{align*}
Once we precompose with $G$, the estimate follows.
\end{proof}

Hypothesis of theorem \ref{teorema inversa} are satisfied; regularity propositions \ref{lipschitz}-\ref{smoothness} guarantee the uniqueness and smoothness of the normal form. Theorem \ref{theorem Moser diffeo} follows.

\subsection{The translated curve of R\"ussmann}
\label{Russmann theorem appendix}
The diffeomorphisms considered by R\"ussmann are of this kind in a neighborhood of $\normal{T}_0$
\begin{equation}
Q(\teta,r)= (\teta + 2\pi\alpha + t(r) + f(\teta,r), (1 +A) r + g(\teta,r)),
\label{Russmann diffeo}
\end{equation}
where $\alpha$ is Diophantine, $t(0)=0$ and $t'(r)>0$ for every $r$. This represents a perturbation of $$P^0 (\teta,r)= (\teta + \alpha + t(r), (1 +A) r),$$ for which $T_0$ is invariant and carries a rotation $2\pi\alpha$.
\begin{thm}[R\"ussmann]
\label{Russmann theorem} Fix $\alpha\in D_{\gamma,\tau}$ and $$P^0 (\teta,r)=(\teta + 2\pi\alpha + t(r) + O(r^2), (1 + A)r + O(r^2))\,\in U(\alpha,A)$$ such that $t(0)=0$ and $t'(r)>0$.\\  If $Q$ is close enough to $P^0$ there exists a unique analytic curve $\gamma\in\A(\T,\R)$, close to $r=0$, a diffeomorphism $\varphi$ of $\T$ close to the identity and $b\in\R$, close to $0$, such that \[Q (\teta, \gamma(\teta)) = (\conj(\teta), b + \gamma(\conj(\teta))).\]
\end{thm}
Actually in its original version the theorem is stated for $A=0$; to consider the more general case with $A$ close to $0$, does not bring any further difficulties. 

To deduce R\"ussmann's theorem from theorem \ref{theorem Moser diffeo} we need to get rid of the counter-terms $\beta$ and $B$.

\subsubsection{Elimination of $B$}
In order to deduce R\"ussmann's result from theorem \ref{theorem Moser diffeo} we need to reduce the number of translation terms of $T_\lambda$ to one, corresponding to the translation in the $r$-direction ($T_{\lambda= (0, b)}$). As we are not interested in keeping the same normal dynamics of the perturbed diffeomorphism $Q$, up to let $A$ vary and conjugate $Q$ to some well chosen $P_{A}$, we can indeed make the counter term $B\cdot r$ to be zero. \\ Let $\Lambda_2 = \set{\lambda = (\beta, b),\, \beta,b\in\R}$.
\begin{prop}
For every $P^0\in U_{s+\sigma}(\alpha,A_0)$ with $\alpha$ diophantine, there is a germ of $C^{\infty}$ maps
\begin{equation*}
\psi: \A(\normal{T}_{s+\sigma},\T_\C\times\C)\to \G_{s}\times U_{s}(\alpha,A)\times\Lambda_2,\quad Q\mapsto (G,P,\lambda),
\end{equation*} at $P^0\mapsto (\id, P^0, 0)$, such that $Q = T_\lambda\circ G\circ P\circ G^{-1}$.
\label{proposition B=0}
\end{prop}   
\begin{proof}
Denote $\phi_A$ the operator $\phi$, as now we want $A$ to vary. Let us write $P^0$ as $$P^0(\teta,r)=(\teta + 2\pi\alpha + O(r), (1 + A_0 - \delta A)\cdot r + \delta A\cdot r + O(r^2)),$$ and remark that \[P^0 = T_\lambda\circ P_{A},\quad \lambda = \pa{0, B\cdot r = (-\delta A + \frac{\delta A\cdot A_0}{1 + A_0})\cdot r}, \] where $ P_A = (\teta + 2\pi\alpha + O(r), (1 + A)\cdot r + O(r^2))$\footnote{The terms $O(r)$ and $O(r^2)$ contain a factor $(1 + \frac{\delta A}{1 + A_0 - \delta A})$} with $$ A = \pa{A_0 + \frac{\delta A(1 + A_0)}{1 + A_0 - \delta A}}.$$  According to theorem \ref{theorem Moser diffeo}, $\phi_A(\id, P_A, \lambda) = P^0.$ In particular \[\frac{\partial B}{\partial \delta A}_{|_{G=\id}}= -\id + \frac{A_0}{1+A_0},\] where $A_0$ is close to $0$. Hence, defining
\begin{equation*}
\hat{\psi}:\R\times \A(\normal{T}_{s+\sigma},\T_\C\times\C)\to \G_{s}\times U_{s}(\alpha,A)\times\Lambda,\, (A,Q)\mapsto \hat{\psi}_A(Q):=\phi^{-1}_A(Q) = (G, P, \lambda)
\end{equation*}
in the neighborhood of $(A_0, P^0)$, by the implicit function theorem locally for all $Q$ there exists a unique $\bar A$ such that $B (\bar A, Q) = 0$. It remains to define $\psi(Q)= \hat{\psi}(\bar A, Q)$.
\end{proof}
Whenever the interest lies on the translation of the curve and the dynamics tangential to it, do not care about the "final" $A$ and consider the situation that puts $B=0$. \\ In particular the graph of $\gamma(\teta):= R_0\circ\varphi^{-1}(\teta)$ is translated by $b$ and its dynamics is conjugated to $R_{2\pi\alpha}$, modulo the term $\beta$: \[Q(\teta, \gamma(\teta)) = (\beta + \conj(\teta), b + \gamma (\conj(\teta))).\]

\subsubsection{A family of translated curves}
Theorem \ref{theorem Moser diffeo} guarantees that any given diffeomorphism $Q$, sufficiently close to $P^0$ (see equation \eqref{P^0}), is of the form $Q = T_\lambda\circ G\circ P\circ G^{-1}$, with $G$, $P$ and $T_\lambda$
uniquely determined, implying the existence of a curve whose image by $Q$ is translated. Actually there exists a whole family of translated curves. Indeed, let us consider a parameter $c\in B_1(0)$ (the unit ball in $\R$) and the family of diffeomorphisms $Q_c (\teta, r) := Q (\teta, c + r)$ relative to the given $Q$. Considering the corresponding normal form operators $\phi_c$, the parametrized version of theorem \ref{theorem Moser diffeo} follows readily.\\
Now, if $Q_c$ is close enough to $P^0,$ proposition \ref{theorem B=0} asserts the existence of $(G_c, P_c, \lambda_c)\in\G\times U(\alpha,A)\times\Lambda_2$ such that \[Q_c = T_\lambda\circ G_c\circ P_c\circ G^{-1}_c.\] 
Hence we have a family of curves parametrized by $\tilde{c}= c + \int_\T\gamma\,\frac{d\teta}{2\pi}$,  $$Q(\teta, \tilde{c} + \tilde\gamma(\teta)) = (\beta + \conj(\teta), b + \tilde{c} + \tilde\gamma (\conj(\teta))),$$ where $\tilde{\gamma} = \gamma - \int_\T\gamma\,\frac{d\teta}{2\pi}.$
 
\subsubsection{Torsion property: elimination of $\beta$}
As we have seen in the last section, under smallness and diophantine conditions on $Q$, there exists a family of curves, parametrized by $c$, whose images are translated by $b$ in the $r$-direction and whose tangential dynamics is conjugated to the rotation $R_{2\pi\alpha}$, modulo the term $\beta\in\R$.\\ In order to get the dynamical conjugacy to the rotation stated by R\"ussmann's theorem, it is of fundamental importance for $Q$ to satisfy some torsion property, and this is provided by the request that $t'(r)>0$ for every $r$. Once this property is satisfied, in the light of the previous section, in order to prove R\"ussmann's theorem is suffices to show that there exists a unique $c$ close to $0$ such that $\beta = \beta(c) = 0$.  \\ We want to show that the map $c\mapsto \beta (c)$ is a local diffeomorphism.\\ It suffices to show this for the trivial perturbation $P^0_c$. The Taylor expansion of $P^0_c$ directly gives $c \mapsto \beta(c) = t(c) + O(c^2)$, which is a local diffeomorphism due to the torsion hypothesis on $Q$. Hence, the analogous map for $Q_c$, is a small perturbation of the previous one, hence a local diffeomorphism too. Then there exists a unique $c\in\R$ such that $\beta(c)= 0$.
\subsubsection{A comment about higher dimension}
In dimension higher than $2$, the analogue of R\"ussmann's theorem could not be possible: needing the matrix $B\in\Mat_m(\R)$, $m\geq 2$, to solve the third homological equation and having just $m$ characteristic exponents of $A$ at our disposal that we may vary as we did in the last sections, it is hopeless to kill the whole $B$. As a consequence, the obtained surface will undergo more than a simple translation. The analogous result is stated as follows.\\

Let $U(\alpha,A)$ be the space of germs of diffeomorphisms along $\normal{T}^n_0\subset \T^n\times\R^m$ of the form \[P(\teta,r)= (\teta + 2\pi\alpha + T(r) + O(r^2), (1 + A)\cdot r + O(r^2)),\] where $A\in\Mat_m(\R)$ is a diagolanizable matrix of real eigenvalues $a_j\neq 0$ for $j=1,\ldots,m$ and $T(r)$ is such that $T(0)=0$ and $T'(r)$ is invertible for all $r\in\R^m$.\\ Let also $\G$ be the space of germs of real analytic isomorphisms of the form $$g(\teta,r) = (\varphi(\teta), R_0(\teta) + R_1(\teta)\cdot r),$$ $\varphi$ being a diffeomorphism of $\T^n$ fixing the origin, $R_0$ and $R_1$ an $\R^m$-valued and $\Mat_m(\R)$-valued functions defined on $\T^n$. \\ Let $\Lambda_{m^2} = \set{\lambda = (0, b + B\cdot r),\, b\in\R^m, B\in\Mat_m(\R)}$, where $B\in\Mat_m(\R)$ has the $(m^2 - m)$ non diagonal entries different from $0$.
\begin{thm}
Let $\alpha$ be Diophantine. If $Q$ is sufficiently close to $P^0\in U(\alpha,A_0)$, there exists a unique $(G,P,\lambda)\in \G\times U(\alpha,A)\times\Lambda_{m^2}$, close to $(\id, P^0, 0)$ such that \[Q = T_\lambda\circ G\circ P \circ G^{-1}.\]
\end{thm}
The proof follows from the generalization in dimension $\geq 2$ of theorem \ref{theorem Moser diffeo} and the elimination of parameters, which are not hard to recover. We just give some guiding remarks.
\begin{itemize}[leftmargin=*]
\item First note that the constant matrix $M$ appearing in the proof of proposition \ref{linear proposition diffeo} gets the form $M=\pa{\mat{\id & 0\\ 0 & \id + B}}$ and that equation \eqref{linear equation diffeo2} splits into three homological equations of the form
\begin{align*}
\dot\varphi(\teta + 2\pi\alpha) - \dot\varphi(\teta) + p_1 \cdot\dot R_0 &= \dot q_0 + \dot\beta \\
\dot R_0(\teta + 2\pi\alpha) - (\id + A)\cdot\dot R_0(\teta) &= \dot Q_0 + \dot b\\
\dot R_1(\teta + 2\pi\alpha)\cdot (\id + A) - (\id + A)\cdot \dot R_1(\teta)  &=
 \dot Q_1 +  \dot B.
\end{align*}
In particular,
\begin{itemize}
\item equation determining $\dot{R}_0$ is readily solved by applying lemma \ref{lemma cohomological circle} component-wise, if $A$ is diagonal. If $A$ is diagonalizable and $P\in\GL_m(\R)$ is such that $P^{-1}A P$ is diagonal, left multiply the equation by $P^{-1}$ and solve it for $\tilde { R}_0 = P^{-1} \dot R_0$ and $\tilde {Q}_0= P^{-1} \dot Q_0$ 
\item equation determining the matrix $\dot R_1$ consists, when $A$ is diagonal of eigenvalues $a_1,\ldots, a_m$, in solving $m$ equations of the form \[(1 + a_j) \pa{\dot R^j_{j}(\teta + 2\pi\alpha) - \dot R^j_{j}(\teta)} = \dot Q^j_{{j}},\quad j=1,\ldots, m\] and $m^2-m$ equations of the form \[(1+ a_j)\dot R^{i}_j(\teta + 2\pi\alpha) - (1 + a_i)\dot R^{i}_{j}(\teta) = \dot Q^i_j(\teta),\quad \forall i\neq j, i,j = 1,\ldots, m.\] If $A$ is diagonalizable, and $P\in\GL_m(\R)$ the transition matrix, left and right multiply the equation by $P^{-1}$ and $P$ respectively, then solve.
\end{itemize}
\item Eventually, the torsion property on $T(r)$ guarantees the elimination of the $\beta$-obstruction to the rotation conjugacy. The analogue of proposition \ref{theorem B=0} holds. 
\end{itemize}
 %%%%% end comment %%%%%%%%%%%
}
%%%%%%%%%%%%%%%%%%%%%%%%%%%%%%%%%%%%%%%%%%%%%%% old sections %%%%%%%%%%%%%%%%%%%%%%%%%%%%%%%%%%%%%%

\subsection*{Acknowledgments} \co{This work would have never seen the light without the mathematical (and moral) support and advises of A. Chenciner and J. Féjoz all the way through my Ph.D. I'm grateful and in debt with both of them. Thank you to T. Castan, B. Fayad, J. Laskar, J.-P. Marco, P. Mastrolia, L. Niederman and P. Robutel for the enlightening discussions we had and their constant interest in this (and future) work.}

%\bibliographystyle{plain}
%\bibliography{finalbiblio}

\nocite{Calleja-Celletti-delaLlave:KAM}
\nocite{Herman:1983}
\nocite{Sevryuk:1999}
\nocite{Sevryuk:2006}
\nocite{Fayad-Krikorian:2009}
\nocite{Herman:1979}
\nocite{Herman:notes}
\nocite{Arnold:1963}
\nocite{Cheng-Sun:1989}
\end{document}